\documentclass[11pt,a4paper,twoside]{article}
\usepackage{graphics}
\usepackage{amssymb}
\usepackage{amsmath}
\usepackage{mathrsfs}
\usepackage{makeidx}
\usepackage{color}
\usepackage{graphicx}
\usepackage{subfigure}
\usepackage{multicol}
\usepackage{multirow}
\usepackage{float}
\usepackage{booktabs}
\usepackage{braket}
\usepackage[colorlinks,
            linkcolor=red,
            anchorcolor=blue,
            citecolor=blue
            ]{hyperref}
\usepackage{caption}
\usepackage{algorithm}
\usepackage{algorithmic}
\usepackage{latexsym, cite, bm, amsthm, amsmath}
\usepackage{amsthm}
\usepackage{marginnote}
\usepackage{mathrsfs}
\usepackage{epsfig}
\usepackage{epstopdf}
\usepackage[T1]{fontenc}

\newtheorem{thm}{Theorem}[section]

\newtheorem{lem}{Lemma}[section]
\newtheorem{prop}{Proposition}[section]
\newtheorem{defn}{Definition}[section]

\newtheorem{rem}{Remark}[section]

\newtheorem{assu}{Assumption}[section]
\newtheorem{exam}{Example}[section]
\numberwithin{equation}{section}

\DeclareMathOperator{\diag}{diag}
\DeclareMathOperator{\sgn}{sgn}
\DeclareMathOperator{\Res}{Res}
\DeclareMathOperator{\Iter}{Iter}
\DeclareMathOperator{\Cpu}{Cpu}
\DeclareMathOperator{\co}{co}
\DeclareMathOperator{\vecv}{vec}

\title{A non-monotone smoothing Newton algorithm for solving the system of generalized absolute value equations}

\usepackage[marginal]{footmisc}
\usepackage{authblk}
\author[a]{Cairong Chen\thanks{Supported partially by the National Natural Science Foundation of China (Grant No. 11901024) and  the Natural Science Foundation of Fujian Province (Grand No. 2021J01661). Email address: cairongchen@fjnu.edu.cn.}}
\author[b]{Dongmei Yu\thanks{Corresponding author. Supported partially by the Natural Science Foundation of Liaoning Province (Nos. 2020-MS-301 and 2019-BS-118), the Liaoning Provincial Department of Education (Nos. LJ2020ZD002 and LJ2019ZL001) and the China Postdoctoral Science Foundation (2019M650449). Email address: yudongmei1113@163.com.}}
\author[c]{Deren Han\thanks{Supported partially by the National Natural Science Foundation of China (Grant Nos. 11625105 and 11926358). Email address: handr@buaa.edu.cn.}}
\author[a]{Changfeng Ma\thanks{Supported partially by the National Key Research and Development Program of China (Nos.
2019YFC0312003). Email address: macf@fjnu.edu.cn.}}
\affil[a]{School of Mathematics and Statistics, FJKLMAA and Center for Applied Mathematics of Fujian Province, Fujian Normal University, Fuzhou, 350007, P.R. China.}
\affil[b]{Institute for Optimization and Decision Analytics, Liaoning Technical University, Fuxin, 123000, P.R. China.}
\affil[c]{LMIB of the Ministry of Education, School of Mathematical Sciences, Beihang University, Beijing, 100191, P.R. China.}
\begin{document}
\date{\today}
\maketitle

%\double

%\begin{center}
%%%%%%%%%%%%%%%%%%%%%%%%%%%%%%%%%%%%%%%%%%%%%%
\begin{quote}
{\bf Abstract:} The system of generalized absolute value equations (GAVE) has attracted more and more attention in the optimization community. In this paper, by introducing a smoothing function, we develop a smoothing Newton algorithm with non-monotone line search to solve the GAVE. We show that the non-monotone algorithm is globally and locally quadratically convergent under a weaker assumption than those given in most existing algorithms for solving the GAVE. Numerical results are given to demonstrate the viability and efficiency of the approach.

{\small
\medskip
{\em 2000 Mathematics Subject Classification}. 65F10, 65H10, 90C30

\medskip
{\em Keywords}.
Generalized absolute value equations; Smoothing function; Smoothing Newton algorithm; Non-monotone line search; Global and local quadratic convergence.
}

\end{quote}
%\end{center}
%%%%%%%%%%%%%%%%%%%%%%%%%%%%%%%%%%%%%%%%%%%%
\section{Introduction}\label{sec:intro}
The system of generalized absolute value equations (GAVE) is to find a vector $x \in \mathbb{R}^n$ such that
\begin{equation}\label{eq:gave}
Ax + B|x| - b = 0\footnote{In the literature, GAVE also occurs in the form of $Ax - B|x| = b$. In this paper, we do not make a distinction between it and \eqref{eq:gave} and put it down to GAVE~\eqref{eq:gave}.},
\end{equation}
where $A\in \mathbb{R}^{n\times n}$ and $0\neq B\in \mathbb{R}^{n\times n}$ are two known matrices, $b\in \mathbb{R}^n$ is a known vector, and $|x|$ denotes the componentwise absolute value of $x\in \mathbb{R}^n$. To the best of our knowledge, GAVE~\eqref{eq:gave} was first introduced by Rohn in \cite{rohn2004} and further investigated in \cite{prok2009,mang2007,hladik2018,mezzadri2020,hu2011,miao2017} and references therein. Obviously, when $B = -I$ with $I$ being the identity matrix, GAVE~\eqref{eq:gave} becomes the system of absolute value equations (AVE)
\begin{equation}\label{eq:ave}
Ax - |x| - b = 0,
\end{equation}
which is the subject of numerous research works; see, e.g., \cite{mang2006,zamani2021,zhang2009,caccetta2011,mang2007a,haghani2015} and references therein.

GAVE~\eqref{eq:gave} and AVE~\eqref{eq:ave} have attracted considerable attention in the field of optimization for almost twenty years, and the primary reason is that they are closely related to the linear complementarity problem (LCP) \cite{mang2006,prok2009} and the horizontal LCP (HLCP) \cite{mezzadri2020}, which encompass many mathematical programming problems and have many practical applications \cite{cottle2009,mezzadri2020m}. In addition, GAVE~\eqref{eq:gave} and AVE~\eqref{eq:ave} are also
bound up with the system of linear interval equations \cite{rohn1989}.

Due to the combinatorial character introduced by the absolute value operator, solving GAVE~\eqref{eq:gave} is generally NP-hard \cite[Proposition~2]{mang2007}.  Moreover, if GAVE~\eqref{eq:gave} is solvable, checking whether it has a unique solution or multiple solutions is NP-complete \cite[Proposition~2.1]{prok2009}. Recently, GAVE~\eqref{eq:gave} and AVE~\eqref{eq:ave} have been extensively investigated in the literature, and the main research effort can be summarized to the following two aspects.

On the theoretical side, one of the main branches is to investigate conditions for existence, non-existence and uniqueness of solutions of GAVE~\eqref{eq:gave} or AVE~\eqref{eq:ave}; see, e.g., \cite{rohn2004,mezzadri2020,wu2019,mang2006,hu2010,prok2009,hladik2018,wu2021,wu2018} and references therein. Specially, the following necessary and sufficient conditions that ensure the existence and uniqueness of solution of GAVE~\eqref{eq:gave} can be found in \cite{mezzadri2020,wu2021} (see section~\ref{sec:pre} for the definition of the column $\mathcal{W}$-property).

\begin{thm}(\cite[Theorem~1]{mezzadri2020})\label{thm:unique}
The following statements are equivalent:
\begin{itemize}
  \item [(i)] GAVE~\eqref{eq:gave} has a unique solution for any $b\in \mathbb{R}^n;$

  \item [(ii)] $\{A+B,A-B\}$ has the column $\mathcal{W}$-property;

  \item [(iii)] for arbitrary nonnegative diagonal matrices $D_1,D_2\in \mathbb{R}^{n\times n}$ with $D_1 + D_2>0,$
      \begin{equation*}
      \det\left[ (A+B)D_1 + (A-B)D_2\right]\neq 0;
      \end{equation*}

  \item [(iv)] $A + B$ is nonsingular and $\{I,(A+B)^{-1}(A-B)\}$ has the column $\mathcal{W}$-property.
\end{itemize}
\end{thm}

\begin{thm}(\cite[Theorem~3.2]{wu2021})\label{thm:unique2}
GAVE~\eqref{eq:gave} has a unique solution for any $b\in \mathbb{R}^n$ if and only if matrix $A + BD$ is nonsingular for any diagonal matrix $D = \diag(d_i)$ with $d_i \in [-1,1]$.
\end{thm}
It is easy to conclude that Theorem~\ref{thm:unique} and Theorem~\ref{thm:unique2} imply that $\{A+B,A-B\}$ has the column $\mathcal{W}$-property if and only if matrix $A + BD$ is nonsingular for any diagonal matrix $D = \diag(d_i)$ with $d_i \in [-1,1]$ (see Lemma~\ref{lem:ns} for more details).

On the numerical side, there are various algorithms for solving AVE~\eqref{eq:ave} or GAVE~\eqref{eq:gave}. For example, Mangasarian proposed concave minimization method \cite{mang2007a}, generalized Newton method \cite{mang2009}, and successive linear programming method \cite{mang2009k}, for solving AVE~\eqref{eq:ave}.
Zamani and Hlad\'{i}k proposed a new concave minimization algorithm for AVE~\eqref{eq:ave} \cite{zamani2021}, which solves a deficiency of the method proposed in \cite{mang2007a}. Zainali and Lotfi modified the generalized Newton method and developed a stable and quadratic convergent method for AVE~\eqref{eq:ave} \cite{zainali2018}. Cruz et al. proposed an inexact semi-smooth Newton method for AVE~\eqref{eq:ave} \cite{cruz2016}. Shahsavari and Ketabchi proposed two types of proximal algorithms to solve AVE~\eqref{eq:ave} \cite{shahsavari2021}. Haghani introduced generalized Traub's method for AVE~\eqref{eq:ave} \cite{haghani2015}. Ke and Ma proposed an SOR-like iteration method for AVE~\eqref{eq:ave} \cite{ke2017}. Caccetta et al. proposed a smoothing Newton method for AVE~\eqref{eq:ave} \cite{caccetta2011}. Saheya et al. summarized several systematic ways of constructing smoothing functions and proposed a unified neural network model for solving AVE~\eqref{eq:ave} \cite{saheya2019}. Zhang and Wei proposed a generalized Newton method which combines the semismooth and the smoothing Newton steps for AVE~\eqref{eq:ave} \cite{zhang2009}. In  \cite{lian2018}, Lian et al. further considered the generalized Newton method for GAVE~\eqref{eq:gave} and presented some weaker convergent conditions compared to the results in \cite{mang2009,hu2011}.  Wang et al. proposed modified Newton-type iteration methods for GAVE~\eqref{eq:gave} \cite{wang2019}. Zhou et al. established Newton-based matrix splitting methods for GAVE~\eqref{eq:gave} \cite{zhou2021}. Jiang and Zhang proposed a smoothing-type algorithm for GAVE~\eqref{eq:gave} \cite{jiang2013}. Tang and Zhou proposed a quadratically convergent descent method for GAVE~\eqref{eq:gave} \cite{tang2019}. Hu et al. proposed a generalized Newton method for absolute value equations associated with second order cones (SOCAVE), which is an extension of GAVE~\eqref{eq:gave} \cite{hu2011}. For more numerical algorithms, one can refer to \cite{achache2018,chen2021,gu2017,edhs2017,yuch2020,ke2020,maer2018} and references therein.

By looking into the mathematical format of GAVE~\eqref{eq:gave}, non-differentiability is caused by the absolute value operator. Smoothing algorithms have been successfully applied to solve GAVE~\eqref{eq:gave} \cite{jiang2013,tang2019}. However, monotone line search techniques were used in the methods proposed in \cite{jiang2013,tang2019}. Recently, great attention has been paid to smoothing algorithms with non-monotone line search; see, e.g., \cite{tang2021,tang2018,huang2009,zhu2011} and references therein. Non-monotone line search schemes can improve the likelihood of finding a global optimum and improve convergence speed in cases where a monotone line search scheme is forced to creep along the bottom of a narrow curved valley \cite{zhang2004}. It is therefore interesting to develop non-monotone smoothing algorithms for solving GAVE~\eqref{eq:gave}. This motivates us to develop a non-monotone smoothing Newton algorithm for solving GAVE~\eqref{eq:gave}. Our work here is inspired by recent studies on weighted complementarity problem \cite{tang2021,tang2021q}.

The rest of this paper is organized as follows. In section~\ref{sec:pre}, we provide some concepts and results used throughout the paper. In section~\ref{sec:nsna}, we develop a non-monotone smoothing Newton algorithm for solving GAVE~\eqref{eq:gave}, while section~\ref{sec:convergence} is devoted to discussing the convergence. Numerical experiments are given in section~\ref{sec:numer}. Finally, section~\ref{sec:conclusion} concludes this paper.

\textbf{Notation.} $\mathbb{R}^{m\times m}$ is the set of all $m\times m$ real matrices,
$\mathbb{R}^n=\mathbb{R}^{n\times 1}$, and
$\mathbb{R}=\mathbb{R}^1$. $\mathbb{R}_+$ and $\mathbb{R}_{++}$ denote the nonnegative and positive real number, respectively.  $I_n$ (or simply $I$ if its dimension
is clear from the context) is the $n\times n$ identity matrix. The
superscript  ``${\cdot}^{\top}$'' takes transpose. For $X\in \mathbb{R}^{m\times n}$, $X_{i,j}$ refers to its $(i,j)$th entry, $|X|$ is in $\mathbb{R}^{m\times n}$ with its $(i,j)$th entry $|X_{i,j}|$. Inequality $X\le Y$ means $X_{i,j}\le Y_{i,j}$ for all $(i,j)$, and similarly for $X<Y$.
We use $t \downarrow 0$ to denote the case that a positive scalar $t$ tends to $0$. We use $\alpha = O(\beta)$ to mean $\frac{\alpha}{\beta}$ is bounded uniformly as $\beta \rightarrow 0$. For any $a\in \mathbb{R}$, we define $\sgn(a) := \left\{\begin{array}{lr} 1,& \text{if}~a >0, \\ 0, & \text{if}~a =0,\\ -1,& \text{if}~a <0. \end{array}\right.$ We denote the diagonal matrix whose $i$th diagonal element is $x_i$ by $\diag(x_i)$ and define $D(x) := \diag(\sgn(x_i))$. The symbol $\|\cdot\|$ stands for the $2$-norm. For a matrix $P\in \mathbb{R}^{m\times n}$, we use $\sigma_{\min}(P)$ and $\sigma_{\max}(P)$ to denote the smallest singular value and the largest singular value, respectively. For a differentiable mapping $G:\mathbb{V}\subset\mathbb{R}^n\rightarrow \mathbb{R}^n$, we denote $G'(x)$ by the Jacobian of $G$ at $x\in \mathbb{V}$ and $\nabla G(x) = G'(x)^\top$ denotes the gradient of $G$ at $x$. For $x\in \mathbb{R}^n$, we also denote $x$ by $\vecv(x_i)$.

\section{Preliminaries}\label{sec:pre}
In this section, we collect some basic notions as well as corresponding assertions, which are useful in this paper.

\begin{defn}(\cite{sznajder1995})
Let $\mathcal{M}:= \{M,N\}$ be a set of matrices $M,\,N\in \mathbb{R}^{n\times n}$, a matrix $R\in \mathbb{R}^{n\times n}$ is called a column representative of $\mathcal{M}$ if
$$
R_{\cdot j} \in \{M_{\cdot j},N_{\cdot j}\},\quad j = 1,2,\cdots,n,
$$
where $R_{\cdot j},\,M_{\cdot j}$ and $N_{\cdot j}$ denote the $j$th column of $R,\,M$ and $N$, respectively. $\mathcal{M}$ is said to have the column $\mathcal{W}$-property if the determinants of all column representative matrices of $\mathcal{M}$ are all positive or all negative.
\end{defn}

\begin{defn}(\cite{rohn1989})
An interval matrix $A^I$ is defined by $A^I := [\underline{A},\bar{A}]=\{X:\underline{A}\le X\le \bar{A}\}$. A square interval matrix $A^I$ is called regular if each $X\in A^I$ is nonsingular.
\end{defn}

\begin{defn}\label{defn:dd}(See, e.g., \cite{fischer1997})
The classic (one-sided) directional derivative of a function $f:\mathbb{R}^n \rightarrow \mathbb{R}$ at $x$ in the direction $y$ is defined by
\begin{equation*}%\label{eq:dd}
f'(x;y) = \lim_{t\downarrow 0}\frac{f(x + ty)-f(x)}{t},
\end{equation*}
provided that the limit exists. Accordingly, $F'(x;y) = [F'_1(x;y), \cdots, F'_m(x;y)]^\top$ denotes the directional derivative for the vector-valued function $F:\mathbb{R}^n \rightarrow \mathbb{R}^m$.
\end{defn}

\begin{defn}(\cite{fischer1997})
A vector-valued function $F:\mathbb{R}^n\rightarrow \mathbb{R}^m$ is said to be Lipschitz continuous on a set $\mathcal{S} \subset \mathbb{R}^n$ if there is a constant $L>0$ such that
\begin{equation*}%\label{ie:lc}
\|F(x) - F(y)\|\le L\|x-y\|,\quad x,y\in S.
\end{equation*}
Moreover, $F$ is called locally Lipschitz continuous on $\mathbb{R}^n$ if it is Lipschitz continuous on all compact subsets $\mathcal{S} \subset \mathbb{R}^n$.
\end{defn}

If $F:\mathbb{R}^n\rightarrow \mathbb{R}^m$ is locally Lipschitz continuous, by Rademacher's Theorem, $F$ is differentiable almost everywhere \cite{qi1993}. Let $D_F$ be the set where $F$ is differentiable, then the generalized Jacobian of $F$ at $x$ in the sense of Clarke \cite{clarke1983} is
$$
\partial F(x) = \co\left\{ \lim_{\begin{array}{c}x^{(k)}\in D_F\\x^{(k)} \rightarrow x\end{array}}\nabla F(x^{(k)})\right\},
$$
where ``$\co$'' denotes the convex hull.

\begin{defn}(\cite{qi2000})
A locally Lipschitz continuous vector-valued function $F:\mathbb{R}^n\rightarrow \mathbb{R}^m$ is called semismooth at $x$ if
$$
\lim_{\begin{array}{c}V\in \partial F(x+td')\\d'\rightarrow d,t\downarrow 0\end{array}} \{Vd'\}
$$
exists for any $d\in \mathbb{R}^n$.
\end{defn}

\begin{lem}(\cite{qi2000}\label{lem:dd})
Let $F:\mathbb{R}^n\rightarrow \mathbb{R}^m$, then the directional derivative $F'(x;d)$ exists for any $d\in \mathbb{R}^n$ if $F$ is semismooth at $x$.
\end{lem}

\begin{lem}(\cite{qi2000}\label{lem:ssm})
Suppose that $F:\mathbb{R}^n\rightarrow \mathbb{R}^m$ is semismooth at $x$. Then it is called strongly semismooth at $x$ if
$$
Vd - F'(x;d) = O(\|d\|^2)
$$
for any $V\in \partial F(x+d)$ and $d\rightarrow 0$.
\end{lem}

Throughout the rest of this paper, we always assume that the following assumption holds.
\begin{assu}\label{assu:as}
Let matrices $A$ and $B$ satisfy
$\{A+B,A-B\}$ has the column $\mathcal{W}$-property.
\end{assu}
It is known that, if Assumption~\ref{assu:as} holds, GAVE~\eqref{eq:gave} has a unique solution for any $b\in \mathbb{R}^n$~\cite{mezzadri2020}. In addition, we have the following lemma, which is needed in the subsequent discussion.

\begin{lem}\label{lem:ns}
Assumption~\ref{assu:as} holds if and only if matrix $A + BD$ is nonsingular for any diagonal matrix $D = \diag (d_i)$ with $d_i\in [-1,1](i=1,2,\cdots,n)$.
\end{lem}
\begin{proof}
The result can be straightly derived from Theorem~\ref{thm:unique} and Theorem~\ref{thm:unique2}. Indeed, it follows from Theorem~\ref{thm:unique} that Assumption~\ref{assu:as} holds if and only if
\begin{equation*}%\label{eq:det}
\det\left[ (A+B)\bar{D} + (A-B)\tilde{D}\right]\neq 0
\end{equation*}
for any nonnegative diagonal matrices $\bar{D},\tilde{D}\in \mathbb{R}^{n\times n}$ with $\bar{D} + \tilde{D} > 0$, that is,
\begin{equation}\label{neq:eq}
\framebox{
\parbox{12cm}{
$\det\left[ A + B(\bar{D} - \tilde{D})(\bar{D} + \tilde{D})^{-1}\right]\neq 0$ for any nonnegative diagonal matrices $\bar{D},\tilde{D}\in \mathbb{R}^{n\times n}$ with $\bar{D} + \tilde{D} > 0$.
}
}
\end{equation}

Let
\begin{align*}
\mathcal{D}_1 &:= \{D\in \mathbb{R}^{n\times n}:D = \diag(d_i), d_i\in [-1,1](i=1,2,\cdots,n)\},\\
\mathcal{D}_2 &:= \{D\in \mathbb{R}^{n\times n}:D = (\bar{D}-\tilde{D})(\bar{D} + \tilde{D})^{-1}, \bar{D} =\diag(\bar{d}_i)\ge 0, \tilde{D} = \diag(\tilde{d}_i)\ge 0, \bar{D} + \tilde{D} >0\}.
\end{align*}
Then, on one hand, for any $D\in \mathcal{D}_2$, we have $|D_{i,i}| = \frac{|\bar{d}_i - \tilde{d}_i|}{|\bar{d}_i + \tilde{d}_i|}\le 1$. Thus, $\mathcal{D}_2 \subseteq \mathcal{D}_1$. On the other hand, for any $D = \diag(d_i)\in \mathcal{D}_1$, $d_i\in [-1,1]$ can be expressed by $d_i = \frac{\bar{d}_i - \tilde{d}_i}{\bar{d}_i + \tilde{d}_i}$ with
\begin{equation*}
\begin{cases}
&\bar{d}_i > 0, \tilde{d}_i = 0,\quad \text{if}~d_i = 1;\\
&\bar{d}_i = 0, \tilde{d}_i > 0,\quad \text{if}~d_i = -1;\\
&\bar{d}_i = \frac{(1+d_i)\tilde{d}_i}{1-d_i}>0, \quad \text{if}~d_i \in (-1,1).
\end{cases}
\end{equation*}
Hence, $\mathcal{D}_1 \subseteq \mathcal{D}_2$. It follows from the above discussion that $\mathcal{D}_1 = \mathcal{D}_2$. Then \eqref{neq:eq} is equivalent to
\begin{equation*}
\det(A + BD)\neq 0
\end{equation*}
for any $D = \diag (d_i)$ with $d_i\in [-1,1](i=1,2,\cdots,n)$. This completes the proof.
\end{proof}

\begin{rem}\label{rem:remark1}
For symmetric matrices $A$ and $B$, under the assumption that $\sigma_{\min}(A) > \sigma_{\max}(B)$, the authors in \cite[Lemma~1]{anane2020} proved the nonsingularity of $A + BD$ for any diagonal matrix $D$ whose elements are equal to $1,0$ or $-1$. We should mentioned that the symmetries of the matrices $A$ and $B$ can be relaxed there and our result here is more general than theirs.
\end{rem}

\begin{rem}
In \cite{jiang2013}, the authors used the assumption that $\sigma_{\min}(A) > \sigma_{\max}(B)$, while in \cite{tang2019}, the authors used the assumption that the interval matrix $[A-|B|, A+|B|]$ is regular. The interval matrix $[A-|B|, A+|B|]$ is regular is weaker than that $\sigma_{\min}(A) > \sigma_{\max}(B)$ and examples can be found in \cite[Examples~2.1~and~2.3]{zhang2009}. In addition, it is easy to prove that $[A-|B|, A+|B|]$ is regular implies that Assumption~\ref{assu:as} holds, but the reverse is not true. For instance, let
$$
A = \left[\begin{array}{cc} 1001 & -496\\ -994& 501\end{array} \right], \quad B = \left[\begin{array}{cc} 999 & -494\\ -995& 499\end{array} \right],
$$
then $\{A+B,A-B\}$ has the column $\mathcal{W}$-property \cite{mezzadri2020} while $[A - |B|, A + |B|]$ is not regular. Indeed, there exists a singular matrix
$
\left[\begin{array}{cc} 2 & -2\\ -2& 2\end{array} \right]\in [A - |B|, A + |B|].
$
In conclusion, our Assumption~\ref{assu:as} here is more general than those used in \cite{jiang2013,tang2019}.
\end{rem}

\section{The algorithm}\label{sec:nsna}
In this section, we develop a non-monotone smoothing Newton algorithm for solving GAVE~\eqref{eq:gave}. To this end, we first consider an equivalent reformulation of GAVE~\eqref{eq:gave} by introducing a smoothing function for the absolute value operator.

\subsection{A smoothing function for $|x|$ with $x\in \mathbb{R}$}\label{subsec:smooth}
In this subsection, we consider a smoothing function for $|x|$ with $x\in \mathbb{R}$ and discuss some of its properties, which lay the foundation of the next subsection.

Since $|x|$ is not differentiable at $x = 0$, in order to overcome the hurdle in analysis and application, researchers construct numerous smoothing functions for it \cite{saheya2019}. In this paper, we adopt the following smoothing function $\phi:\mathbb{R}^2\rightarrow \mathbb{R}$, defined by
\begin{equation}\label{eq:smooth}
\phi(\mu,x) = \sqrt{\mu^2 + x^2} - \mu,
\end{equation}
which can be derived from the perspective of the convex conjugate \cite{saheya2019}.

In the following, we give some properties related to the smoothing function \eqref{eq:smooth}.

\begin{prop}\label{prop:prop1}
Let $\phi$ be defined by \eqref{eq:smooth}, then we have
\begin{itemize}
  \item [(i)] $\phi(0,x) = |x|;$

  \item [(ii)] $\phi$ is continuously differentiable on $\mathbb{R}^2\setminus \{(0,0)\}$, and when $(\mu,x)\neq (0,0)$, we have
      $$
      \frac{\partial \phi}{\partial \mu} = \frac{\mu}{\sqrt{\mu^2 + x^2}}-1 \quad \text{and} \quad \frac{\partial \phi}{\partial x} = \frac{x}{\sqrt{\mu^2 + x^2}};
      $$

  \item [(iii)] $\phi$ is a convex function on $\mathbb{R}^2$, i.e., $\phi(\alpha(\bar{\mu},\bar{x}) + (1-\alpha)(\tilde{\mu},\tilde{x})) \le \alpha\phi(\bar{\mu},\bar{x}) + (1-\alpha)\phi(\tilde{\mu},\tilde{x})$ for all $(\bar{\mu},\bar{x}),(\tilde{\mu},\tilde{x})\in \mathbb{R}^2$ and $\alpha \in [0,1];$

  \item [(iv)] $\phi$ is Lipschitz continuous on $\mathbb{R}^2;$

  \item [(v)] $\phi$ is strongly semismooth on $\mathbb{R}^2$.
\end{itemize}
\end{prop}
\begin{proof}
The proofs of (i) and (ii) are trivial.

Now we turn to the result (iii). For any $(\bar{\mu},\bar{x}),(\tilde{\mu},\tilde{x})\in \mathbb{R}^2$ and $\alpha \in [0,1]$, we have
\begin{align}\nonumber
&\phi(\alpha(\bar{\mu},\bar{x}) + (1-\alpha)(\tilde{\mu},\tilde{x})) - \alpha\phi(\bar{\mu},\bar{x}) - (1-\alpha)\phi(\tilde{\mu},\tilde{x})\\\nonumber
&\qquad =\sqrt{[\alpha\bar{\mu} + (1-\alpha)\tilde{\mu}]^2 + [\alpha \bar{x} + (1-\alpha)\tilde{x}]^2} - \alpha\bar{\mu} - (1-\alpha)\tilde{\mu}\\\nonumber
&\qquad \qquad  - \alpha\sqrt{\bar{\mu}^2 + \bar{x}^2} + \alpha\bar{\mu} -(1-\alpha)\sqrt{\tilde{\mu}^2 + \tilde{x}^2} + (1-\alpha)\tilde{\mu}\\\label{eq:convex}
& \qquad = \sqrt{[\alpha\bar{\mu} + (1-\alpha)\tilde{\mu}]^2 + [\alpha \bar{x} + (1-\alpha)\tilde{x}]^2} - \alpha\sqrt{\bar{\mu}^2 + \bar{x}^2}-(1-\alpha)\sqrt{\tilde{\mu}^2 + \tilde{x}^2}.
\end{align}
On one hand,
\begin{align}\nonumber
&\left(\sqrt{[\alpha\bar{\mu} + (1-\alpha)\tilde{\mu}]^2 + [\alpha \bar{x} + (1-\alpha)\tilde{x}]^2}\right)^2 \\\label{eq:lhs}
&\qquad = \alpha^2(\bar{\mu}^2+ \bar{x}^2) + (1-\alpha)^2(\tilde{\mu}^2 + \tilde{x}^2) + 2\alpha(1-\alpha)(\bar{\mu}\tilde{\mu} + \bar{x}\tilde{x}).
\end{align}
On the other hand,
\begin{align}\nonumber
&\left[\alpha\sqrt{\bar{\mu}^2 + \bar{x}^2} + (1-\alpha)\sqrt{\tilde{\mu}^2 + \tilde{x}^2}\right]^2\\\nonumber
 &\qquad = \alpha^2(\bar{\mu}^2+ \bar{x}^2) + (1-\alpha)^2(\tilde{\mu}^2 + \tilde{x}^2) +  2\alpha(1-\alpha) \sqrt{(\bar{\mu}^2 + \bar{x}^2)(\tilde{\mu}^2 + \tilde{x}^2)}\\\label{eq:rhs}
&\qquad \ge \alpha^2(\bar{\mu}^2+ \bar{x}^2) + (1-\alpha)^2(\tilde{\mu}^2 + \tilde{x}^2) + 2\alpha(1-\alpha)|\bar{\mu}\tilde{\mu} + \bar{x}\tilde{x}|.
\end{align}
Then the result (iii) follows from \eqref{eq:convex}-\eqref{eq:rhs}.

Consider the result (iv). For any $(\bar{\mu},\bar{x}),(\tilde{\mu},\tilde{x})\in \mathbb{R}^2$, we have
\begin{align*}
|\phi(\bar{\mu},\bar{x}) - \phi(\tilde{\mu},\tilde{x})|&= |\|(\bar{\mu},\bar{x})\| - \bar{\mu} - \|(\tilde{\mu},\tilde{x})\| + \tilde{\mu}|\\
&\le |\|(\bar{\mu},\bar{x})\| - \|(\tilde{\mu},\tilde{x})\|| + |\bar{\mu}-\tilde{\mu}|\\
&\le \|(\bar{\mu}-\tilde{\mu},\bar{x}-\tilde{x})\| + \|(\bar{\mu}-\tilde{\mu},\bar{x}-\tilde{x})\| \\
&= 2 \| (\bar{\mu}-\tilde{\mu},\bar{x}-\tilde{x})\|.
\end{align*}
Hence, $\phi$ is Lipschitz continuous with Lipschitz constant $2$.

Finally, we prove the result (v). It follows from the result (iii) that $\phi$ is semismooth on $\mathbb{R}^2$~\cite{qi2000}. Note that $\phi$ is arbitrarily many times differentiable for all $(\mu,x)\in \mathbb{R}^2$ with $(\mu,x)\neq (0,0)$ and hence strongly semismooth at these points. Therefore, it is sufficient to show that it is strongly semismooth at $(0,0)$. For any $(\mu,x)\in \mathbb{R}^2\backslash \{(0,0)\}$, $\phi$ is differentiable at $(\mu,x)$, and hence, $\partial\phi(\mu,x) = \nabla \phi(\mu,x) = \left[\frac{\partial \phi(\mu,x)}{\partial \mu},\frac{\partial \phi(\mu,x)}{\partial x}\right]^\top$. In addition, by Lemma~\ref{lem:dd}, the classic directional derivative of $\phi$ at $(0,0)$ exists and
$$
\phi'((0,0);(\mu,x)) = \lim_{t\downarrow 0} \frac{\phi((0,0) + t(\mu,x)) - \phi(0,0)}{t}
=\phi(\mu,x),
$$
from which we have
\begin{align*}
&\phi(\mu,x) - \left[\frac{\partial \phi(\mu,x)}{\partial \mu},\frac{\partial \phi(\mu,x)}{\partial x}\right]\left[\begin{array}{c}\mu\\x\end{array}\right]\\
&\qquad = \sqrt{\mu^2 + x^2} - \mu - \left(\frac{\mu}{\sqrt{\mu^2 + x^2}}-1\right)\mu - \frac{x}{\sqrt{\mu^2 + x^2}}x\\
&\qquad = 0\\
&\qquad = O(\|(\mu,x)\|^2).
\end{align*}
Then the result follows from Lemma~\ref{lem:ssm}.
\end{proof}

\subsection{The reformulation of GAVE~\eqref{eq:gave}}\label{subsec:reformu}
In this subsection, based on the earlier subsection, we will give a reformulation of GAVE~\eqref{eq:gave} and explore some of its properties.

Let $z := (\mu,x) \in \mathbb{R}\times \mathbb{R}^{n}$, we first define the function $H:\mathbb{R}\times \mathbb{R}^n \rightarrow \mathbb{R}\times \mathbb{R}^n$ as
\begin{equation}\label{eq:sf4ave}
H(z) := \left[ \begin{array}{c} \mu \\
Ax + B\Phi(\mu,x) - b \end{array}\right],
\end{equation}
where $\Phi:\mathbb{R}^{n+1} \rightarrow \mathbb{R}^n$ is defined by
\begin{equation*}
\Phi(\mu,x) := \left[ \begin{array}{c} \phi(\mu,x_1) \\
\phi(\mu,x_2)\\
\vdots\\
\phi(\mu,x_n) \end{array}\right]
\end{equation*}
with $\phi$ being the smoothing function given in \eqref{eq:smooth}. According to Proposition~\ref{prop:prop1}~(i), it holds that
\begin{equation}\label{eq:gavevsne}
H(z) = 0\quad \Leftrightarrow\quad \mu=0~\text{and}~x~\text{is a solution of GAVE~\eqref{eq:gave}}.
\end{equation}
Then it follows from \eqref{eq:gavevsne} that solving GAVE~\eqref{eq:gave} is equivalent to solving the system of nonlinear equations $H(z) = 0$. Before giving the algorithm for solving $H(z) = 0$, we will give some properties of the function $H$.

\begin{prop}\label{prop:prop2}
Let $H$ be defined by \eqref{eq:sf4ave}, then we have
\begin{itemize}
  \item [(i)] $H$ is continuously differentiable on $\mathbb{R}^{n+1}\backslash \{0\}$, and when $\mu=0$ and $x_i\neq 0$ (for all $i=1,2,\cdots,n$) or $\mu\neq 0$, the Jacobian matrix of $H$ is given by
      \begin{equation}\label{eq:jh}
H'(z) = \left[\begin{array}{cc}
1& 0\\
BV_1 & A + BV_2
\end{array}\right]
\end{equation}
with
\begin{equation}\label{eq:vd}
V_1 = \left[\begin{array}{c}
\frac{\mu}{\sqrt{\mu^2 + x_1^2}} - 1\\
\frac{\mu}{\sqrt{\mu^2 + x_2^2}} - 1\\
\vdots\\
\frac{\mu}{\sqrt{\mu^2 + x_n^2}} - 1
\end{array}\right], \quad
V_2 = \left[\begin{array}{cccc}
\frac{x_1}{\sqrt{\mu^2 + x_1^2}} & 0 & 0 & 0\\
0 & \frac{x_2}{\sqrt{\mu^2 + x_2^2}} &0 &0\\
\vdots &\vdots &\ddots &\vdots\\
0&0&0&\frac{x_n}{\sqrt{\mu^2 + x_n^2} }
\end{array}\right];
\end{equation}

  \item [(ii)] $H$ is strongly semismooth on $\mathbb{R}^{n+1}.$
\end{itemize}
\end{prop}
\begin{proof}
The result (i) holds from Proposition~\ref{prop:prop1}~(ii).

Now we turn to prove the result (ii). Since $H$ is strongly semismooth on $\mathbb{R}^{n+1}$ if and only if its component function $H_i$, $i=1,2,\cdots,n$, are \cite{qi2000}, and the composition of strongly semismooth functions is a strongly semismooth function \cite[Theorem~19]{fischer1997}, the result (ii) follows from Proposition~\ref{prop:prop1}~(v) and the fact that a continuously differentiable function with a Lipschitz continuous gradient is strongly semismooth \cite{hu2009}.
\end{proof}

\subsection{The non-monotone smoothing Newton algorithm for GAVE~\eqref{eq:gave}}\label{subsuc:alg}
Now we are in position to develop a non-monotone smoothing Newton algorithm to solve the system of nonlinear equations $H(z) = 0$, and so is GAVE~\eqref{eq:gave}.

Let $H(z)$ be given in \eqref{eq:sf4ave} and define the merit function $\mathcal{M}:\mathbb{R}\times \mathbb{R}^n \rightarrow \mathbb{R}_+$ by
\begin{equation*}%\label{eq:merit}
\mathcal{M}(z) := \|H(z)\|^2.
\end{equation*}
Clearly, solving the system of nonlinear equations $H(z) = 0$ is equivalent to solving the following unconstrained optimization problem
\begin{equation*}%\label{eq:optim}
\min_{z\in \mathbb{R}^{n+1}}~\mathcal{M}(z)
\end{equation*}
with the vanished objective function value. We now propose a non-monotone smoothing Newton algorithm to solve $H(z) = 0$ by minimizing the merit function $\mathcal{M}(z)$, which is described in Algorithm~\ref{alg:NSNA}.

\begin{algorithm}
\caption{A non-monotone smoothing Newton algorithm (NSNA) for GAVE~\eqref{eq:gave}}
\label{alg:NSNA}
\begin{algorithmic}[1]
\STATE Choose $\theta,\,\delta \in (0,1)$ and $z^{(0)}:= (\mu^{(0)},x^{(0)})\in \mathbb{R}_{++}\times \mathbb{R}^n$. Let $\mathcal{C}^{(0)}:= \mathcal{M}(z^{(0)})$. Choose $\gamma \in (0,1)$ such that $\beta^{(0)} = \gamma\mathcal{C}^{(0)}< \mu^{(0)}$ and $\gamma\mu^{(0)}<1$. Set $k:= 0$.

\STATE If $\|H(z^{(k)})\|=0$, then stop. Else, compute the search direction $\Delta z^{(k)} = (\Delta \mu^{(k)}, \Delta x^{(k)})\in \mathbb{R}\times \mathbb{R}^n$ by solving the perturbed Newton system:
\begin{equation}\label{eq:ns}
H'(z^{(k)}) \Delta z^{(k)} = - H(z^{(k)}) + \beta^{(k)} e^{(1)},
\end{equation}
where $e^{(1)} = [1,0]^\top\in \mathbb{R}\times \mathbb{R}^n$. If $\Delta z^{(k)}$ satisfies
\begin{equation}\label{ie:hh}
\|H(z^{(k)} + \Delta z^{(k)}) \| \le \theta \|H(z^{(k)})\|,
\end{equation}
then set $z^{(k+1)} := z^{(k)} + \Delta z^{(k)}$ and go to step~$4$. Otherwise, go to step~$3$.

\STATE Let $\alpha^{(k)}$ be the maximum of the values $1,\delta,\delta^2,\cdots$ such that
\begin{equation}\label{ie:nml}
\mathcal{M}(z^{(k)} + \alpha^{(k)} \Delta z^{(k)}) \le \mathcal{C}^{(k)} - \gamma \|\alpha^{(k)} \Delta z^{(k)}\|^2.
\end{equation}
Set $z^{(k+1)} := z^{(k)} + \alpha^{(k)} \Delta z^{(k)}$ and go to step~$4$.

\STATE Compute $\mathcal{M}(z^{(k+1)}) = \|H(z^{(k+1)})\|^2$ and set
\begin{equation}\label{eq:ck}
\mathcal{C}^{(k+1)} := \frac{(\mathcal{C}^{(k)} + 1)\mathcal{M}(z^{(k+1)})}{\mathcal{M}(z^{(k+1)}) + 1},\quad \beta^{(k+1)} := \gamma \mathcal{C}^{(k+1)}.
\end{equation}

\STATE Set $k:= k+1$ and go to step $2$.
\end{algorithmic}
\end{algorithm}

\begin{rem}
The development of Algorithm~\ref{alg:NSNA} is inspired by the non-monotone smoothing Newton algorithm for the weighted complementarity problem \cite{tang2021} and the non-monotone Levenberg-Marquardt type method for the weighted nonlinear complementarity problem \cite{tang2021q}.
\end{rem}

Before ending this section, we will show that Algorithm~\ref{alg:NSNA} is well-defined. To this end, we need the following lemma.

\begin{lem}\label{thm:ns}
Let $H'(z)$ be defined by \eqref{eq:jh} and \eqref{eq:vd}. If Assumption~\ref{assu:as} holds, then $H'(z)$ is nonsingular at any $z=(\mu,x)\in \mathbb{R}_{++}\times \mathbb{R}^n$.
\end{lem}
\begin{proof}
From \eqref{eq:jh}, we need only to show that $A + BV_2$ is nonsingular. Since Assumption~\ref{assu:as} holds and $\left|\frac{x_i}{\sqrt{x_i^2 + \mu^2}}\right|<1 (i = 1,2,\cdots,n)$, the result immediately follows from Lemma~\ref{lem:ns}.
\end{proof}

Then we have the following theorem.

\begin{thm}\label{thm:wd}
If Assumption~\ref{assu:as} holds, Algorithm~\ref{alg:NSNA} is well defined and either terminates in finitely many steps or generates an infinite sequence $\{z^{(k)}\}$ satisfying $\mathcal{M}(z^{(k)}) \le \mathcal{C}^{(k)}$, $\mu^{(k)} >0$ and $\beta^{(k)} <\mu^{(k)}$ for all $k\ge 0$.
\end{thm}
\begin{proof}
We will prove it by mathematical induction. Suppose that $\mathcal{M}(z^{(k)}) \le \mathcal{C}^{(k)}$, $\mu^{(k)} >0$ and $\beta^{(k)} <\mu^{(k)}$ for some $k$. Since $\mu^{(k)} >0$, it follows from Lemma~\ref{thm:ns} that $H'(z^{(k)})$ is nonsingular. Hence, $\Delta z^{(k)}$ can be uniquely determined by \eqref{eq:ns}.
If $\|H(z^{(k)})\|=0$, then Algorithm~\ref{alg:NSNA} terminates. Otherwise, $\|H(z^{(k)})\|^2 = \mathcal{M}(z^{(k)}) \le \mathcal{C}^{(k)}$ implies that $\mathcal{C}^{(k)}>0$, from which and the second equation in \eqref{eq:ck} we have $\beta^{(k)} = \gamma \mathcal{C}^{(k)}>0$. In the following, we divide our proof in three parts.

Firstly, we will show that $\mu^{(k+1)}>0$. On one hand, if $z^{(k+1)}$ is generated by step $2$, it follows from \eqref{eq:ns} that $\mu^{(k+1)} = \mu^{(k)} + \Delta \mu^{(k)} = \mu^{(k)} + (-\mu^{(k)} + \beta^{(k)}) = \beta^{(k)} >0$. On the other hand, if $z^{(k+1)}$ is generated by step $3$, we first show that there exists at least a nonnegative integer $l$ satisfying~\eqref{ie:nml}. On the contrary, for any nonnegative integer $l$, we have
\begin{equation}\label{ie:cnml}
\mathcal{M}(z^{(k)} + \delta^l \Delta z^{(k)}) > \mathcal{C}^{(k)} - \gamma \|\delta^l \Delta z^{(k)}\|^2,
\end{equation}
which together with $\mathcal{M}(z^{(k)}) \le \mathcal{C}^{(k)}$ gives
\begin{equation*}
\frac{\mathcal{M}(z^{(k)} + \delta^l \Delta z^{(k)})-\mathcal{M}(z^{(k)} )}{\delta^l} + \gamma \delta^l\| \Delta z^{(k)}\|^2>0.
\end{equation*}
Since $\mathcal{M}$ is differentiable at $z^{(k)}$ and $\delta\in (0,1)$, by letting $l\rightarrow +\infty$ in the above inequality, we have
\begin{equation}\label{ie:mdz}
\mathcal{M}'(z^{(k)} )\Delta z^{(k)} \ge 0.
\end{equation}
In addition, from \eqref{eq:ns} we have
\begin{align}\nonumber
\mathcal{M}'(z^{(k)} )\Delta z^{(k)} &= 2 H(z^{(k)})^\top H'(z^{(k)}) \Delta z^{(k)} \\\nonumber
&= -2\|H(z^{(k)})\|^2 + 2\mu^{(k)} \beta^{(k)}  \\\label{eq:ms}
&= 2\mu^{(k)}(\beta^{(k)}-\mu^{(k)}) - 2 \|Ax^{(k)} + B\Phi(\mu^{(k)},x^{(k)})-b\|^2.
\end{align}
Since $\mu^{(k)} >0$ and $\beta^{(k)} <\mu^{(k)}$,  \eqref{eq:ms} implies that $\mathcal{M}'(z^{(k)} )\Delta z^{(k)}<0$, which contradicts to~\eqref{ie:mdz}. Therefore, there exists $\alpha^{(k)}\in (0,1]$ such that $z^{(k+1)} = z^{(k)} + \alpha^{(k)} \Delta z^{(k)}$ in step $3$. In this case, it follows from~\eqref{eq:ns} that $\mu^{(k+1)} = (1-\alpha^{(k)})\mu^{(k)} + \alpha^{(k)} \beta^{(k)} >0$.

Secondly, we will show that $\mathcal{M}(z^{(k+1)}) < \mathcal{C}^{(k+1)}$. Indeed, if $z^{(k+1)}$ is generated by step~$2$, then it follows from $\theta \in (0,1)$ and \eqref{ie:hh} that $\mathcal{M}(z^{(k+1)}) \le \theta^2 \mathcal{M}(z^{(k)}) < \mathcal{M}(z^{(k)}) \le \mathcal{C}^{(k)}$. Otherwise, by step $3$, we can also obtain $\mathcal{M}(z^{(k+1)})< \mathcal{C}^{(k)}$. In fact, $\beta^{(k)} < \mu^{(k)}$ implies that $\Delta z^{(k)} \neq 0$. Thereby, \eqref{ie:nml} implies that $\mathcal{M}(z^{(k+1)})< \mathcal{C}^{(k)}$. Consequently, $\mathcal{M}(z^{(k+1)})< \mathcal{C}^{(k)}$ and the first equation in \eqref{eq:ck} imply
$$
\mathcal{C}^{(k+1)} = \frac{(\mathcal{C}^{(k)} + 1)\mathcal{M}(z^{(k+1)})}{\mathcal{M}(z^{(k+1)}) + 1}> \frac{(\mathcal{M}^{(k+1)} + 1)\mathcal{M}(z^{(k+1)})}{\mathcal{M}(z^{(k+1)}) + 1} = \mathcal{M}(z^{(k+1)}).
$$

Finally, we will show  that $\mu^{(k+1)} > \beta^{(k+1)}$. As mentioned earlier, we have $\mu^{(k+1)} = \beta^{(k)}$ by step~$2$ and $\mu^{(k+1)} = (1-\alpha^{(k)})\mu^{(k)} + \alpha^{(k)} \beta^{(k)}$ by step~$3$, respectively. For the latter, since $\alpha^{(k)}\in (0,1]$ and $\mu^{(k)} >\beta^{(k)} >0$, $\mu^{(k+1)} = (1-\alpha^{(k)})\mu^{(k)} + \alpha^{(k)} \beta^{(k)} \ge (1-\alpha^{(k)})\beta^{(k)} + \alpha^{(k)} \beta^{(k)} = \beta^{(k)}$. In a word, $\mu^{(k+1)} \ge \beta^{(k)}$. In addition, it follows from $\mathcal{M}(z^{(k+1)}) < \mathcal{C}^{(k)}$ and the first equation in \eqref{eq:ck} that $$0\le \mathcal{C}^{(k+1)} =\frac{\mathcal{C}^{(k)}\mathcal{M}(z^{(k+1)}) + \mathcal{M}(z^{(k+1)})}{\mathcal{M}(z^{(k+1)}) + 1}< \frac{\mathcal{C}^{(k)}\mathcal{M}(z^{(k+1)}) + \mathcal{C}^{(k)}}{\mathcal{M}(z^{(k+1)}) + 1} = \mathcal{C}^{(k)},$$
from which and $\gamma>0$ we obtain $\mu^{(k+1)}\ge \beta^{(k)} = \gamma\mathcal{C}^{(k)} > \gamma\mathcal{C}^{(k+1)} = \beta^{(k+1)}$.

The proof is completed by letting $\mathcal{M}(z^{(0)}) \le \mathcal{C}^{(0)}$, $\mu^{(0)} >0$ and $\beta^{(0)} <\mu^{(0)}$.
\end{proof}

\begin{rem}
We should mention that the equation \eqref{eq:ms} plays the key role in the proof of Theorem~\ref{thm:wd}. This equation motivates us to develop the algorithm with the property $\beta^{(k)} <\mu^{(k)}$, which is slightly different from that given in \cite{tang2021} ($\beta^{(k)} \le \mu^{(k)}$ was proved there).
\end{rem}

\section{Convergence analysis}\label{sec:convergence}
In this section, we will analyze the convergence of Algorithm~\ref{alg:NSNA}. In what follows, we assume that $\|H(z^{(k)})\|\neq 0$ for all $k\ge 0$. To establish the global convergence of Algorithm~\ref{alg:NSNA}, we need the following lemmas.

\begin{lem}\label{lem:cc}
Suppose that Assumption~\ref{assu:as} holds. Let $\{z^{(k)} =(\mu^{(k)},x^{(k)})\}$ be the iteration sequence generated by Algorithm~\ref{alg:NSNA}. Then $\mathcal{C}^{(k)} >\mathcal{C}^{(k+1)}$ and $\mu^{(k)} >\mu^{(k+1)}$ for all $k\ge 0$.
\end{lem}
\begin{proof}
The proof of $\mathcal{C}^{(k)} >\mathcal{C}^{(k+1)}$ for all $k\ge 0$ can be found in the proof of Theorem~\ref{thm:wd}. It follows from Theorem~\ref{thm:wd} that $\mu^{(k)} > \beta^{(k)}$ for all $k\ge 0$. Then, by step~$2$, we have $\mu^{(k+1)} = \beta^{(k)} < \mu^{(k)}$. By step~$3$, $\mu^{(k+1)} = (1-\alpha^{(k)})\mu^{(k)} + \alpha^{(k)} \beta^{(k)} < (1-\alpha^{(k)})\mu^{(k)} + \alpha^{(k)} \mu^{(k)} = \mu^{(k)}$.
The proof is complete.
\end{proof}

\begin{lem}\label{lem:levelbounded}
If Assumption~\ref{assu:as} holds, then $\{z^{(k)}\}$ generated by Algorithm~\ref{alg:NSNA} is bounded.
\end{lem}
\begin{proof}
We first prove that the level set
$$
\mathcal{L}(\Lambda) := \{z = (\mu,x)\in \mathbb{R}^{n+1}: \|H(z)\|\le \Lambda\}
$$
is bounded for any $\Lambda >0$. On the contrary, there exists a sequence $\{\bar{z}^{(k)} = (\bar{\mu}^{(k)},\bar{x}^{(k)})\}$ such that $\lim\limits_{k\rightarrow \infty}\|\bar{z}^{(k)}\| = \infty$ and $\|H(\bar{z}^{(k)})\|\le \Lambda$, where $\Lambda >0$ is some constant. Since
\begin{equation}\label{ie:sh}
\|H(\bar{z}^{(k)})\|^2 = (\bar{\mu}^{(k)})^2 + \|A\bar{x}^{(k)} + B\Phi(\bar{\mu}^{(k)},\bar{x}^{(k)}) - b\|^2,
\end{equation}
we can conclude that $\{\bar{\mu}^{(k)}\}$ is bounded. It follows from this and the unboundedness of $\{(\bar{\mu}^{(k)},\bar{x}^{(k)})\}$ that $\lim\limits_{k\rightarrow \infty}\|\bar{x}^{(k)}\| = \infty$. Since the sequence $\left\{\frac{\bar{x}^{(k)}}{\|\bar{x}^{(k)}\|} \right\}$ is bounded, it has at least one accumulation point $\hat{x}$. Then, there exists a subsequence $\{\bar{x}^{(k)}\}_{k\in K}$ such that $\lim\limits_{k\in K,k\rightarrow +\infty}\frac{\bar{x}^{(k)}}{\|\bar{x}^{(k)}\|} = \hat{x}$ with $K\subset \{0,1,2,\cdots\}$. It follows from the continuity of the $2$-norm that $\|\hat{x}\| = 1$. In the following, we remain $k\in K$. From \eqref{ie:sh}, we have
\begin{align}\nonumber
\frac{\Lambda^2}{\|\bar{x}^{(k)}\|^2} &\ge \frac{\|H(\bar{z}^{(k)})\|^2}{\|\bar{x}^{(k)}\|^2}\\\label{ie:up}
&= \frac{(\bar{\mu}^{(k)})^2}{\|\bar{x}^{(k)}\|^2} + \left\| A\frac{\bar{x}^{(k)}}{\|\bar{x}^{(k)}\|} + B \frac{\Phi(\bar{\mu}^{(k)},\bar{x}^{(k)})}{\|\bar{x}^{(k)}\|} - \frac{b}{\|\bar{x}^{(k)}\|}\right\|^2.
\end{align}
Since
$$\frac{\sqrt{(\bar{\mu}^{(k)})^2 + (\bar{x}^{(k)}_i)^2} - \bar{\mu}^{(k)}}{\|\bar{x}^{(k)}\|} = \sqrt{\left(\frac{\bar{\mu}^{(k)}}{\|\bar{x}^{(k)}\|}\right)^2 + \left(\frac{\bar{x}^{(k)}_i}{\|\bar{x}^{(k)}\|}\right)^2} - \frac{\bar{\mu}^{k}}{\|\bar{x}^{(k)}\|} (i=1,2,\cdots,n),$$
from the boundedness of $\{\bar{\mu}^{(k)}\}$, we have
$$
\lim_{k \rightarrow \infty} \frac{\sqrt{(\bar{\mu}^{(k)})^2 + (\bar{x}^{(k)}_i)^2} - \bar{\mu}^{(k)}}{\|\bar{x}^{(k)}\|} = \sqrt{\hat{x}_i^2} = |\hat{x}_i|.
$$
Hence, by letting $k \rightarrow \infty$ in \eqref{ie:up}, we have $A\hat{x} + B|\hat{x}| = 0$, i.e., $[A + BD(\hat{x})]\hat{x} = 0$. Since $D(\hat{x})\in [-I,I]$, it follows from Lemma~\ref{lem:nonsingular} that $A + BD(\hat{x})$ is nonsingular. Thus, we have $\hat{x} = 0$, which contradicts to the fact that $\|\hat{x}\| = 1$.

If $\{z^{(k)}\}$ is generated by Algorithm~\ref{alg:NSNA}, then $\|H(z^{(k)})\| \le \sqrt{\mathcal{C}^{(0)}}$ for all $k\ge 0$. Hence, $\{z^{(k)}\}$ is bounded based on the aforementioned disscussion.
\end{proof}

\begin{rem}
The proof of Lemma~\ref{lem:levelbounded} is inspired by that of \cite[Theorem~2.3]{tang2019}, which was considered in the case that the interval matrix $[A-|B|, A+|B|]$ is regular. In addition, similar to the proof of \cite[Lemma~4.1]{jiang2013}, the boundedness of $\{z^{(k)}\}$ can be derived under the assumption that $\sigma_{\min}(A)>\sigma_{\max}(B)$. Our result here seems more general than those in \cite[Lemma~4.1]{jiang2013}.
\end{rem}

Now we show the global convergence of Algorithm~\ref{alg:NSNA}.
\begin{thm}\label{thm:gc}
Assume that Assumption~\ref{assu:as} holds. Let $\{z^{(k)} =(\mu^{(k)},x^{(k)})\}$ be the iteration sequence generated by Algorithm~\ref{alg:NSNA}. Then any accumulation point $z^*$ of $\{z^{(k)}\}$ satisfies
$$
H(z^*) = 0,
$$
i.e., $z^* =(0,x^*)$ and $x^*$ is a solution of GAVE~\eqref{eq:gave}.
\end{thm}
\begin{proof}
Lemma~\ref{lem:levelbounded} implies the existence of the accumulation point of $\{z^{(k)}\}$ generated by Algorithm~\ref{alg:NSNA}. Let $z^*$ be any accumulation point of $\{z^{(k)}\}$, then there exists a subsequence of $\{z^{(k)}\}$ converging to $z^*$. For convenience, we still denote the subsequence by $\{z^{(k)}\}$.

By Lemma~\ref{lem:cc}, $\{\mathcal{C}^{(k)}\}$ is convergent because it is monotonically decreasing. Thus, there exists a constant $\mathcal{C}^*\ge 0$ such that $ \lim\limits_{k\rightarrow +\infty}\mathcal{C}^{(k)} = \mathcal{C}^*$. As $\mathcal{M}(z^{(k)}) = \|H(z^{(k)})\|^2 \le \mathcal{C}^{(k)}$ for all $k\ge 0$, $\lim\limits_{k\rightarrow +\infty} \|H(z^{(k)})\| =0$ provided that  $\mathcal{C}^*=0$. Then, from the continuity of $H(z)$ we have $H(z^*) = 0$. In the following, we assume that $\mathcal{C}^*>0$ and derive a contradiction.

According to the first equation in \eqref{eq:ck}, we have
\begin{equation}\label{eq:lm}
\lim_{k\rightarrow +\infty}\mathcal{M}(z^{(k+1)}) = \lim_{k\rightarrow +\infty}\left(\frac{\mathcal{C}^{(k+1)}}{1+\mathcal{C}^{(k)} -\mathcal{C}^{(k+1)}}\right)= \mathcal{C}^*>0.
\end{equation}
By the fact that $\beta^{(k)} = \gamma \mathcal{C}^{(k)}$, we have $\beta^* = \lim\limits_{k\rightarrow +\infty}\beta^{(k)} = \gamma \mathcal{C}^* >0$. Based on Theorem~\ref{thm:wd} and Lemma~\ref{lem:cc}, we have $\mu^* = \lim\limits_{k\rightarrow +\infty}\mu^{(k)} \ge \lim\limits_{k\rightarrow +\infty}\beta^{(k)} = \beta^* >0$. Since $\mu^*>0$, $H'(z^*)$ is nonsingular and $\mathcal{M}$ is continuously differentiable at $z^*$.

Let $\mathcal{N}:= \{k:\|H(z^{(k)} + \Delta z^{(k)})\|\le \theta \|H(z^{(k)})\|\}$. We claim that $\mathcal{N}$ must be a finite set. In fact, if $\mathcal{N}$ is an infinite set, then $\|H(z^{(k)} + \Delta z^{(k)})\|\le \theta \|H(z^{(k)})\|$, i.e., $\mathcal{M}(z^{(k+1)})\le \theta^2 \mathcal{M}(z^{(k)})$ holds for infinitely many $k$. By letting $k\rightarrow +\infty$ with $k\in \mathcal{N}$, we have $\mathcal{C}^* \le \theta^2 \mathcal{C}^*$. This leads to a contradiction due to $\theta\in (0,1)$ and $\mathcal{C}^*>0$. Hence, we can suppose that there exists an index $\bar{k}>0$ such that $\|H(z^{(k)} + \Delta z^{(k)})\|> \theta \|H(z^{(k)})\|$ for all $k\ge \bar{k}$. Then, for all $k\ge \bar{k}$, $z^{(k+1)} = z^{(k)} + \alpha^{(k)} \Delta z^{(k)}$ (generated by step~$3$) satisfies $\mathcal{M}(z^{(k+1)}) \le \mathcal{C}^{(k)} - \gamma \|\alpha^{(k)} \Delta z^{(k)}\|^2$, i.e.,
$$
\gamma \|\alpha^{(k)} \Delta z^{(k)}\|^2 \le \mathcal{C}^{(k)} - \mathcal{M}(z^{(k+1)}),
$$
from which and \eqref{eq:lm} we have $\lim\limits_{k\rightarrow +\infty} \alpha^{(k)}\|\Delta z^{(k)}\| = 0$.

On one hand, if $1\ge \alpha^{(k)} = \delta^{l_k} \ge \varrho >0$ for all $k\ge \bar{k}$ with $\varrho$ being a fixed constant, then $\Delta z^*= \lim\limits_{\bar{k}\le k\rightarrow +\infty} \Delta z^{(k)} = 0$, which implies that
\begin{equation}\label{ie:me}
\mathcal{M}'(z^*)\Delta z^* = 0.
\end{equation}
Here and in the sequel, $\Delta z^*$ is the unique solution of $H'(z^*)\Delta z^* = -H(z^*) + \beta^*e^{(1)}$.

On the other hand, $\{\alpha^{(k)}\}_{k\ge \bar{k}}$ has a subsequence converging to $0$. Without loss of generality, we may assume that $\lim\limits_{\bar{k}\le k\rightarrow +\infty} \alpha^{(k)} =0$. Let $\hat{\alpha}^{(k)} := \delta^{l_k}/\delta$, then $\lim\limits_{\bar{k}\le k\rightarrow +\infty} \hat{\alpha}^{(k)} =0$. Moreover, for all $k\ge \bar{k}$, it follows from the definition of $\alpha^{(k)}$ and Theorem~\ref{thm:wd} that
$$
\mathcal{M}(z^{(k)} + \hat{\alpha}^{(k)} \Delta z^{(k)}) > \mathcal{C}^{(k)} - \gamma \|\hat{\alpha}^{(k)}\Delta z^{(k)}\|^2 \ge \mathcal{M}(z^{(k)}) - \gamma \|\hat{\alpha}^{(k)} \Delta z^{(k)}\|^2.
$$
Thus,
$$
\frac{\mathcal{M}(z^{(k)} + \hat{\alpha}^{(k)} \Delta z^{(k)}) - \mathcal{M}(z^{(k)})}{\hat{\alpha}^{(k)}} + \gamma \hat{\alpha}^{(k)} \|\Delta z^{(k)}\|^2>0.
$$
By letting $k\rightarrow +\infty$ in the above inequality, we have
\begin{equation}\label{ie:mg}
\mathcal{M}'(z^*)\Delta z^* \ge 0.
\end{equation}

Since $\mu^*\gamma\le \gamma \mu^{(0)}<1$, it follows from \eqref{eq:ns} and \eqref{eq:lm}  that
\begin{equation*}%\label{ie:ml}
\frac{1}{2} \mathcal{M}'(z^*)\Delta z^* = H(z^*)^\top H'(z^*)\Delta z^* = - \mathcal{M}(z^*) + \mu^*\beta^* = -\mathcal{C}^* + \mu^*\gamma \mathcal{C}^* = (\mu^*\gamma- 1)\mathcal{C}^*<0,
\end{equation*}
which is contrary to \eqref{ie:me} and \eqref{ie:mg}. The proof of the theorem is now complete.
\end{proof}

Under Assumption~\ref{assu:as}, GAVE~\eqref{eq:gave} has a unique solution and thus Lemma~\ref{lem:levelbounded} and Theorem~\ref{thm:gc} imply that the sequence generated by Algorithm~\ref{alg:NSNA} has a unique accumulation $z^*$ and $\lim\limits_{k\rightarrow +\infty}z^{(k)} = z^*$. In the following, we will discuss the local quadratic convergence of Algorithm~\ref{alg:NSNA}.

\begin{lem}\label{lem:nonsingular}
Assume that Assumption~\ref{assu:as} holds and $z^* = (0,x^*)$ is the accumulation point of the sequence $\{z^{(k)}\}$ generated by Algorithm~\ref{alg:NSNA}. We have:
\begin{itemize}
  \item [(i)] $
  \partial H(z^*) \subseteq \left\{V: V = \left[\begin{array}{cc} 1 &0 \\ B\vecv(\kappa_i-1) &  A + B \diag(\chi_i)
  \end{array}\right],\, \kappa_i, \chi_i \in [-1,1]\,(i = 1,2,\cdots,n)\right\};
  $

  \item [(ii)] All $V\in \partial H(z^*)$ are nonsingular;

  \item [(iii)] There exists a neighborhood $\mathscr{N}(z^*)$ of $z^*$ and a constant $C>0$ such that for any $z:=(\mu,x)\in \mathscr{N}(z^*)$ with $\mu >0$, $H'(z)$ is nonsingular and $\|H'(z)^{-1}\|\le C$.
\end{itemize}
\end{lem}
\begin{proof}
A direct computation yields the result (i). The result (ii) follows from (i) and Lemma~\ref{lem:ns}, and the result (iii) follows from \cite[Lemma~2.6]{qi1993}.
\end{proof}

Owing to Proposition~\ref{prop:prop2}~(ii) and Lemma~\ref{lem:nonsingular}, we can obtain the following local quadratic convergence theorem of Algorithm~\ref{alg:NSNA}. The theorem was well known in the application of smoothing-type Newton methods. The theorem as a whole can be implied by \cite[Theorem~8]{tang2021} and thus we omit the proof here.

\begin{thm}
Assume that Assumption~\ref{assu:as} holds and $z^*$ is the accumulation point of the sequence $\{z^{(k)}\}$ generated by Algorithm~\ref{alg:NSNA}. Then the whole sequence $\{z^{(k)}\}$ converges to $z^*$ with
$$
\|z^{(k+1)}-z^*\| = O(\|z^{(k)} -z^*\|^2).
$$
\end{thm}

\section{Numerical results}\label{sec:numer}
In this section, we will present two numerical examples to illustrate the performance of Algorithm~\ref{alg:NSNA}. Three algorithms will be tested, i.e., Algorithm~\ref{alg:NSNA} (denoted by ``NSNA''), the monotone smoothing Newton algorithm proposed by Jiang and Zhang \cite{jiang2013} (denoted by ``JZ-MSNA'') and the monotone smoothing Newton algorithm proposed by Tang and Zhou \cite{tang2019} (denoted by ``TZ-MSNA''). All experiments are implemented in MATLAB R2018b with a machine precision $2.22\times 10^{-16}$ on a PC Windows 10 operating system with an Intel i7-9700 CPU and 8GB RAM.

We will apply the aforementioned algorithms to solve GAVE~\eqref{eq:gave} arising from HLCP. Given $M,N\in \mathbb{R}^{n\times n}$ and $q\in \mathbb{R}^n$, HLCP is to find a pair $(z,w)\in \mathbb{R}^{n}\times \mathbb{R}^n$ such that
\begin{equation}\label{eq:hlcp}
Mz - Nw = q, \quad z\ge 0, \quad w\ge 0, \quad z^\top w = 0.
\end{equation}
The equivalent relationship between GAVE~\eqref{eq:gave} and HLCP~\eqref{eq:hlcp} can be found in \cite[Proposition~1]{mezzadri2020}.

\begin{exam}\label{exam:exam1}
Consider HLCP~~\eqref{eq:hlcp} with $M = \hat{A} + \xi I$ and $N = \hat{B} + \zeta I$ \cite{mezzadri2020m}, where
$$
\hat{A} = \left[\begin{array}{cccccc}
S & -I & & & & \\
-I & S & -I & & & \\
 & -I & S & -I &&\\
 && \ddots & \ddots &\ddots &\\
 &&& -I & S &-I\\
 &&&& -I & S
 \end{array}\right],\quad \hat{B} = \left[\begin{array}{cccccc}
S &  & & & & \\
 & S &  & & & \\
 &  & S &  &&\\
 && & \ddots & &\\
 &&&  & S &\\
 &&&&  & S
 \end{array}\right]
$$
and
$$
S = \left[\begin{array}{cccccc}
4 & -1 & & & & \\
-1 & 4 & -1 & & & \\
 & -1 & 4 & -1 &&\\
 && \ddots & \ddots &\ddots &\\
 &&& -1 & 4 &-1\\
 &&&& -1 & 4
 \end{array}\right].
$$
\end{exam}

\begin{exam}\label{exam:exam2}
Consider HLCP~~\eqref{eq:hlcp} with $M = \hat{A} + \xi I$ and $N = \hat{B} + \zeta I$ \cite{mezzadri2020m}, where
$$
\hat{A} = \left[\begin{array}{cccccc}
S & -0.5 I & & & & \\
-1.5I & S & -0.5I & & & \\
 & -1.5I & S & -0.5I &&\\
 && \ddots & \ddots &\ddots &\\
 &&& -1.5I & S &-0.5I\\
 &&&& -1.5I & S
 \end{array}\right],\quad \hat{B} = \left[\begin{array}{cccccc}
S &  & & & & \\
 & S &  & & & \\
 &  & S &  &&\\
 && & \ddots & &\\
 &&&  & S &\\
 &&&&  & S
 \end{array}\right]
$$
and
$$
S = \left[\begin{array}{cccccc}
4 & -0.5 & & & & \\
-1.5 & 4 & -0.5 & & & \\
 & -1.5 & 4 & -0.5 &&\\
 && \ddots & \ddots &\ddots &\\
 &&& -1.5 & 4 &-0.5\\
 &&&& -1.5 & 4
 \end{array}\right].
$$
\end{exam}

Obviously, when $\xi,\zeta\ge 0$, matrices $M$ and $N$ in Example~\ref{exam:exam1} are symmetric positive definite while the corresponding matrices in Example~\ref{exam:exam2} are nonsymmetric positive definite. Moreover, it is easy to verify that $\{M,N\}$ has the column $\mathcal{W}$-property \cite{mezzadri2020m}, and thus HLCP~\eqref{eq:hlcp} has a unique solution for any $q\in \mathbb{R}^n$ \cite[Theorem~2]{sznajder1995}. Correspondingly, GAVE~\eqref{eq:gave} with $A = M + N$ and $B = M-N$ satisfies Assumption~\ref{assu:as} and has a unique solution for any $b= q \in \mathbb{R}^n$.

In both examples, we define $q = Mz^* - Nw^*$ with
$$
z^* = [0,1,0,1,\cdots,0,1]^\top,\quad w^* = [1,0,1,0\cdots,1,0]^\top.
$$
In addition, three sets of values of $\xi$ and $\zeta$ are used, i.e., $(\xi,\zeta) =(0,0)$,  $(\xi,\zeta) = (0,4)$ and $(\xi,\zeta) = (4,0)$.

For NSNA, we set $\theta = 0.2, \delta = 0.8, \mu^{(0)} = 0.01$ and choose $\gamma := \min\{\frac{\mu^{(0)}}{\mathcal{C}^{(0)} + 1}, \frac{1}{\mu^{(0)} + 1}, 10^{-12}\}$ such that $\gamma \mathcal{C}^{(0)}<\mu^{(0)}$, $\gamma\mu^{(0)} <1$ and $\gamma\in (0,1)$. For JZ-MSNA, we set $\delta = 0.8, \sigma = 0.2, \mu^{(0)} = 0.01, p = 2$ and choose $\beta := \max\{100, 1.01*(\tau^{(0)})^2/\mu_0\}$ to satisfy the conditions needed for this algorithm \cite{jiang2013} (we refer to \cite{jiang2013} for the definition of $\tau^{(0)}$). For TZ-MSNA, as in \cite{tang2019}, we set $\sigma = 0.2, \delta = 0.8, \gamma = 0.001$ and $\mu^{(0)} = 0.01$. For all methods, $x^{(0)} = [2,2,\cdots,2]^\top$ and methods are stopped if $\Res = \|Ax^{(k)} + B|x^{(k)}| - b\|\le 10^{-7}$ or the maximum number of iteration step $it\_\max = 100$ is exceeded.

For Example~\ref{exam:exam1}, numerical results are shown in Tables~\ref{table1}-\ref{table3}, from which we can find that NSNA is better than JZ-MSNA and TZ-MSNA in terms of $\Iter$ (the number of iterations) and $\Cpu$ (the elapsed CPU time in seconds). Figure~\ref{Fig:Curves-for-P1} plots the convergence curves of the tested methods, from which the monotone convergence properties of all methods are shown\footnote{For JZ-MSNA and TZ-MSNA, $\|H(z^{(k)})\|$ is defined as in \eqref{eq:sf4ave} with $\phi(a,b) = (|a|^p + |b|^p)^{\frac{1}{p}}$.}. For Example~\ref{exam:exam2}, numerical results are shown in Tables~\ref{table4}-\ref{table6}, from which we can also find that NSNA is superior to JZ-MSNA and TZ-MSNA in terms of $\Iter$ and $\Cpu$. Figure~\ref{Fig:Curves-for-P2} plots the convergence curves of the tested methods, from which the monotone convergence properties of JZ-MSNA and TZ-MSNA are shown and the nonmonotone convergence property of NSNA occurs. In conclusion, under our setting, NSNA is a competitive method for solving GAVE~\eqref{eq:gave}.

%%%%%%%%%%%%%%  Tables %%%%%%%%%%%%%%%%%%%%%%%%%%%%
\setlength{\tabcolsep}{1.5pt}
\begin{table}[!h]\small
	\centering
	\caption{Numerical results for Example~\ref{exam:exam1} with $\xi = \zeta = 0$.}\label{table1}
	\begin{tabular}{|c|c|cccc|}\hline
		\multirow{2}{*}{Method}  &  & \multicolumn{4}{|c|}{$n$}  \\ \cline{3-6}
		&     &  $256$ & $1024$ & $2304$ & $4096$ \\\cline{2-6} \hline\hline
		\multirow{3}{*}{NSNA} & $\Iter$  & $\textbf{5}$ & $\textbf{5}$ & $\textbf{6}$ & $\textbf{6}$ \\
		& $\Cpu$ & $ \textbf{0.0044}$ & $\textbf{0.0619}$ & $\textbf{0.3364}$ & $\textbf{0.9985}$ \\
		& $\Res$ & $1.1392\times 10^{-14}$ & $5.7894\times 10^{-14}$  & $2.8847\times 10^{-13}$ &  $1.5603\times 10^{-12}$\\\hline
		\multirow{3}{*}{JZ-MSNA} & $\Iter$  & $6$ & $6$ & $8$ & $7$\\
		& $\Cpu$ & $0.0052$ & $0.0724$ & $0.4497$ & $1.2020$\\
		& $\Res$ &$2.3488\times 10^{-11}$ & $4.9355\times 10^{-10}$ & $1.4026\times 10^{-11}$ & $9.4826\times 10^{-11}$\\\hline
\multirow{3}{*}{TZ-MSNA} & $\Iter$  & $6$ & $6$ & $7$ & $7$\\
		& $\Cpu$ & $0.0049$ & $0.0664$ & $0.3860$ & $1.1631$\\
		& $\Res$ &$1.1147\times 10^{-14}$ & $2.2960\times 10^{-14}$ & $3.4131\times 10^{-14}$ & $4.6847\times 10^{-14}$\\\hline
	\end{tabular}
\end{table}

\setlength{\tabcolsep}{1.5pt}
\begin{table}[!h]\small
	\centering
	\caption{Numerical results for Example~\ref{exam:exam1} with $\xi =0$ and $\zeta = 4$.}\label{table2}
	\begin{tabular}{|c|c|cccc|}\hline
		\multirow{2}{*}{Method}  &  & \multicolumn{4}{|c|}{$n$}  \\ \cline{3-6}
		&     &  $256$ & $1024$ & $2304$ & $4096$ \\\cline{2-6} \hline\hline
		\multirow{3}{*}{NSNA} & $\Iter$  & $\textbf{5}$ & $\textbf{6}$ & $\textbf{7}$ & $\textbf{7}$ \\
		& $\Cpu$ & $ \textbf{0.0035}$ & $\textbf{0.0745}$ & $\textbf{0.3975}$ & $\textbf{1.3254}$ \\
		& $\Res$ & $2.9913\times 10^{-14}$ & $7.8455\times 10^{-13}$  & $4.8343\times 10^{-12}$ &  $2.3106\times 10^{-11}$\\\hline
		\multirow{3}{*}{JZ-MSNA} & $\Iter$  & $7$ & $8$ & $8$ & $9$\\
		& $\Cpu$ & $0.0055$ & $0.1102$ & $0.4676$ & $1.6890$\\
		& $\Res$ &$6.4061\times 10^{-11}$ & $1.7803\times 10^{-9}$ & $4.4746\times 10^{-8}$ & $7.1534\times 10^{-8}$\\\hline
\multirow{3}{*}{TZ-MSNA} & $\Iter$  & $6$ & $7$ & $8$ & $8$\\
		& $\Cpu$ & $0.0039$ & $0.0862$ & $0.4599$ & $1.5145$\\
		& $\Res$ &$1.7297\times 10^{-14}$ & $3.4071\times 10^{-14}$ & $5.3415\times 10^{-14}$ & $6.7748\times 10^{-14}$\\\hline
	\end{tabular}
\end{table}

\setlength{\tabcolsep}{1.5pt}
\begin{table}[!h]\small
	\centering
	\caption{Numerical results for Example~\ref{exam:exam1} with $\xi =4$ and $\zeta = 0$.}\label{table3}
	\begin{tabular}{|c|c|cccc|}\hline
		\multirow{2}{*}{Method}  &  & \multicolumn{4}{|c|}{$n$}  \\ \cline{3-6}
		&     &  $256$ & $1024$ & $2304$ & $4096$ \\\cline{2-6} \hline\hline
		\multirow{3}{*}{NSNA} &$\Iter$  & $\textbf{3}$ & $\textbf{3}$ & $\textbf{3}$ & $\textbf{3}$ \\
		& $\Cpu$ & $ \textbf{0.0021}$ & $\textbf{0.0361}$ & $\textbf{0.1536}$ & $\textbf{0.5444}$ \\
		&$\Res$ & $2.0696\times 10^{-13}$ & $5.8026\times 10^{-12}$  & $4.2044\times 10^{-11}$ &  $1.7292\times 10^{-10}$\\\hline
		\multirow{3}{*}{JZ-MSNA} &$\Iter$  & $4$ & $4$ & $5$ & $5$\\
		& $\Cpu$ & $0.0035$ & $0.0477$ & $0.2846$ & $0.9388$\\
		& $\Res$ &$2.5682\times 10^{-11}$ & $1.1439\times 10^{-11}$ & $2.1336\times 10^{-11}$ & $3.0094\times 10^{-11}$\\\hline
\multirow{3}{*}{TZ-MSNA} & $\Iter$  & $4$ & $4$ & $4$ & $4$\\
		& $\Cpu$ & $0.0024$ & $0.0458$ & $0.2293$ & $0.7017$\\
		& $\Res$ &$1.7786\times 10^{-14}$ & $3.4123\times 10^{-14}$ & $5.2927\times 10^{-14}$ & $6.8855\times 10^{-14}$\\\hline
	\end{tabular}
\end{table}

%%%%%%%%%%%%%%%%%%%%%%%%%%%%%%%%%%%%%%%%%%%%%%%%%%%%%%%%%

\setlength{\tabcolsep}{1.5pt}
\begin{table}[!h]\small
	\centering
	\caption{Numerical results for Example~\ref{exam:exam2} with $\xi = \zeta = 0$.}\label{table4}
	\begin{tabular}{|c|c|cccc|}\hline
		\multirow{2}{*}{Method}  &  & \multicolumn{4}{|c|}{$n$}  \\ \cline{3-6}
		&     &  $256$ & $1024$ & $2304$ & $4096$ \\\cline{2-6} \hline\hline
		\multirow{3}{*}{NSNA} & $\Iter$ & $\textbf{4}$ & $\textbf{5}$ & $\textbf{6}$ & $\textbf{6}$ \\
		& $\Cpu$ & $ \textbf{0.0031}$ & $\textbf{0.0642}$ & $\textbf{0.3332}$ & $\textbf{1.0440}$ \\
		& $\Res$ & $1.0500\times 10^{-14}$ & $5.9914\times 10^{-14}$  & $2.9611\times 10^{-13}$ &  $1.6331\times 10^{-12}$\\\hline
		\multirow{3}{*}{JZ-MSNA} & $\Iter$  & $5$ & $6$ & $8$ & $7$\\
		& $\Cpu$ & $0.0036$ & $0.0774$ & $0.4503$ & $1.2214$\\
		& $\Res$ &$4.1569\times 10^{-11}$ & $2.7899\times 10^{-10}$ & $1.4829\times 10^{-11}$ & $6.0501\times 10^{-9}$\\\hline
\multirow{3}{*}{TZ-MSNA} & $\Iter$  & $5$ & $6$ & $7$ & $7$\\
		& $\Cpu$ & $0.0035$ & $0.0679$ & $0.3776$ & $1.2130$\\
		& $\Res$ &$1.1200\times 10^{-14}$ & $2.0773\times 10^{-14}$ & $3.1480\times 10^{-14}$ & $4.2238\times 10^{-14}$\\\hline
	\end{tabular}
\end{table}

\setlength{\tabcolsep}{1.5pt}
\begin{table}[!h]\small
	\centering
	\caption{Numerical results for Example~\ref{exam:exam2} with $\xi =0$ and $\zeta = 4$.}\label{table5}
	\begin{tabular}{|c|c|cccc|}\hline
		\multirow{2}{*}{Method}  &  & \multicolumn{4}{|c|}{$n$}  \\ \cline{3-6}
		&     &  $256$ & $1024$ & $2304$ & $4096$ \\\cline{2-6} \hline\hline
		\multirow{3}{*}{NSNA} & $\Iter$  & $\textbf{6}$ & $\textbf{7}$ & $\textbf{7}$ & $\textbf{8}$ \\
		& $\Cpu$ & $ \textbf{0.0047}$ & $\textbf{0.0869}$ & $\textbf{0.4074}$ & $\textbf{1.4394}$ \\
		& $\Res$ & $3.0087\times 10^{-14}$ & $7.9704\times 10^{-13}$  & $5.8680\times 10^{-12}$ &  $2.4058\times 10^{-11}$\\\hline
		\multirow{3}{*}{JZ-MSNA} & $\Iter$  & $8$ & $9$ & $10$ & $11$\\
		& $\Cpu$ & $0.0075$ & $0.1225$ & $0.5793$ & $1.9709$\\
		& $\Res$ &$1.4866\times 10^{-11}$ & $2.0476\times 10^{-9}$ & $1.3748\times 10^{-8}$ & $7.0290\times 10^{-8}$\\\hline
\multirow{3}{*}{TZ-MSNA} & $\Iter$  & $7$ & $8$ & $9$ & $9$\\
		& $\Cpu$ & $0.0050$ & $0.0945$ & $0.5193$ & $1.6262$\\
		& $\Res$ &$1.6717\times 10^{-14}$ & $3.3155\times 10^{-14}$ & $5.1814\times 10^{-14}$ & $9.2842\times 10^{-14}$\\\hline
	\end{tabular}
\end{table}

\setlength{\tabcolsep}{1.5pt}
\begin{table}[!h]\small
	\centering
	\caption{Numerical results for Example~\ref{exam:exam2} with $\xi =4$ and $\zeta = 0$.}\label{table6}
	\begin{tabular}{|c|c|cccc|}\hline
		\multirow{2}{*}{Method}  &  & \multicolumn{4}{|c|}{$n$}  \\ \cline{3-6}
		&     &  $256$ & $1024$ & $2304$ & $4096$ \\\cline{2-6} \hline\hline
		\multirow{3}{*}{NSNA} & $\Iter$  & $\textbf{3}$ & $\textbf{3}$ & $\textbf{3}$ & $\textbf{3}$ \\
		& $\Cpu$ & $ \textbf{0.0024}$ & $\textbf{0.0388}$ & $\textbf{0.1623}$ & $\textbf{0.5437}$ \\
		& $\Res$ & $2.1175\times 10^{-13}$ & $5.8378\times 10^{-12}$  & $4.2243\times 10^{-11}$ &  $1.7356\times 10^{-10}$\\\hline
		\multirow{3}{*}{JZ-MSNA} & $\Iter$  & $4$ & $4$ & $5$ & $5$\\
		& $\Cpu$ & $0.0030$ & $0.0500$ & $0.2979$ & $0.8752$\\
		& $\Res$ &$2.5874\times 10^{-11}$ & $1.1674\times 10^{-11}$ & $2.1543\times 10^{-11}$ & $3.0534\times 10^{-11}$\\\hline
\multirow{3}{*}{TZ-MSNA} & $\Iter$  & $4$ & $4$ & $4$ & $4$\\
		&$\Cpu$ & $0.0029$ & $0.0493$ & $0.2448$ & $0.7257$\\
		& $\Res$ &$1.7719\times 10^{-14}$ & $3.3493\times 10^{-14}$ & $4.6185\times 10^{-14}$ & $6.7869\times 10^{-14}$\\\hline
	\end{tabular}
\end{table}

%%%%%%%%%%%%%%%%%%%%% figures %%%%%%%%%%%%%%%%%%%%%%%
\begin{figure}[t]
{\centering
\begin{tabular}{ccc}
\hspace{-0.3 cm}
\resizebox*{0.3\textwidth}{0.26\textheight}{\includegraphics{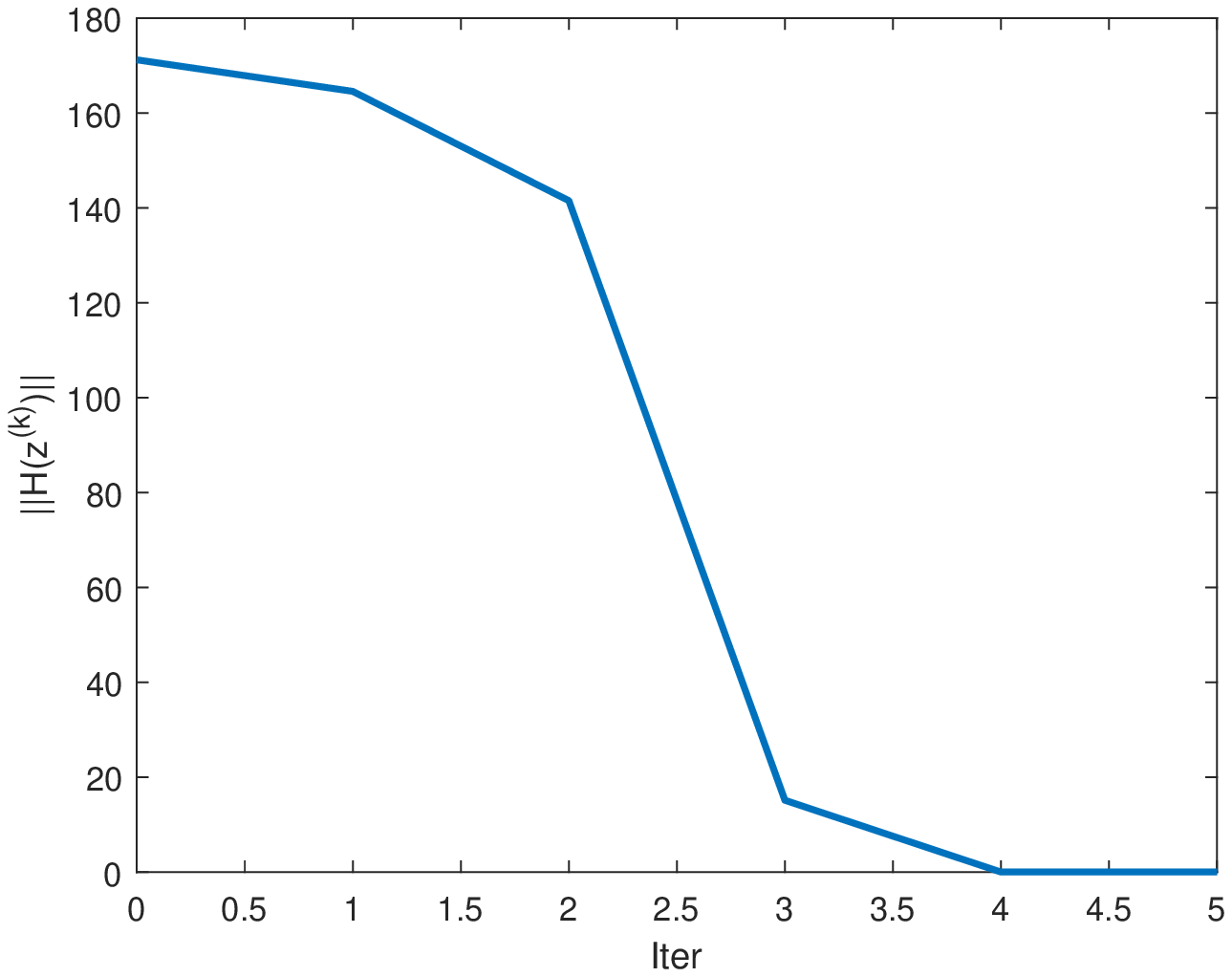}}
&\hspace{-0.4 cm}\resizebox*{0.4\textwidth}{0.26\textheight}{\includegraphics{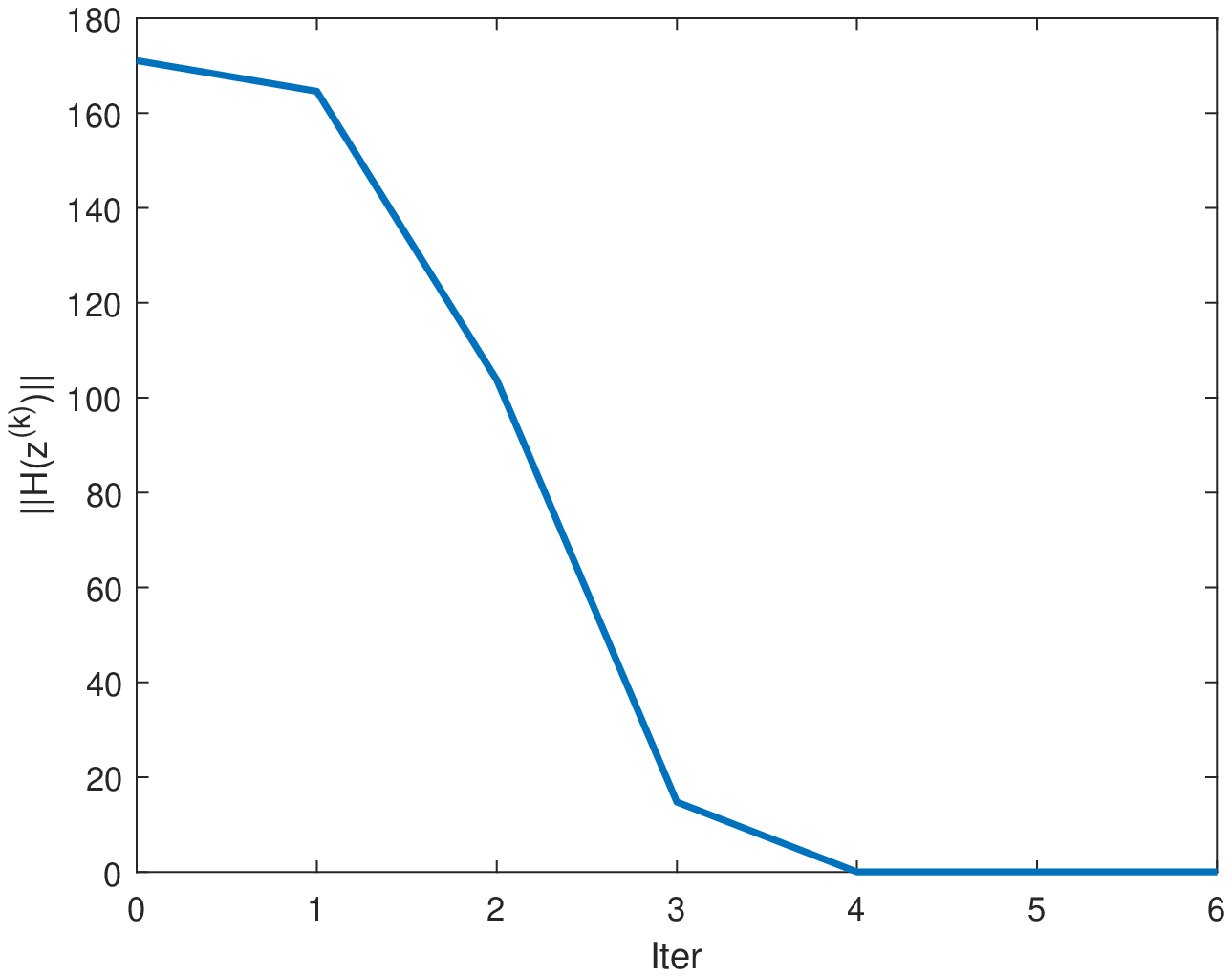}} & \hspace{-0.4 cm}
\resizebox*{0.3\textwidth}{0.26\textheight}{\includegraphics{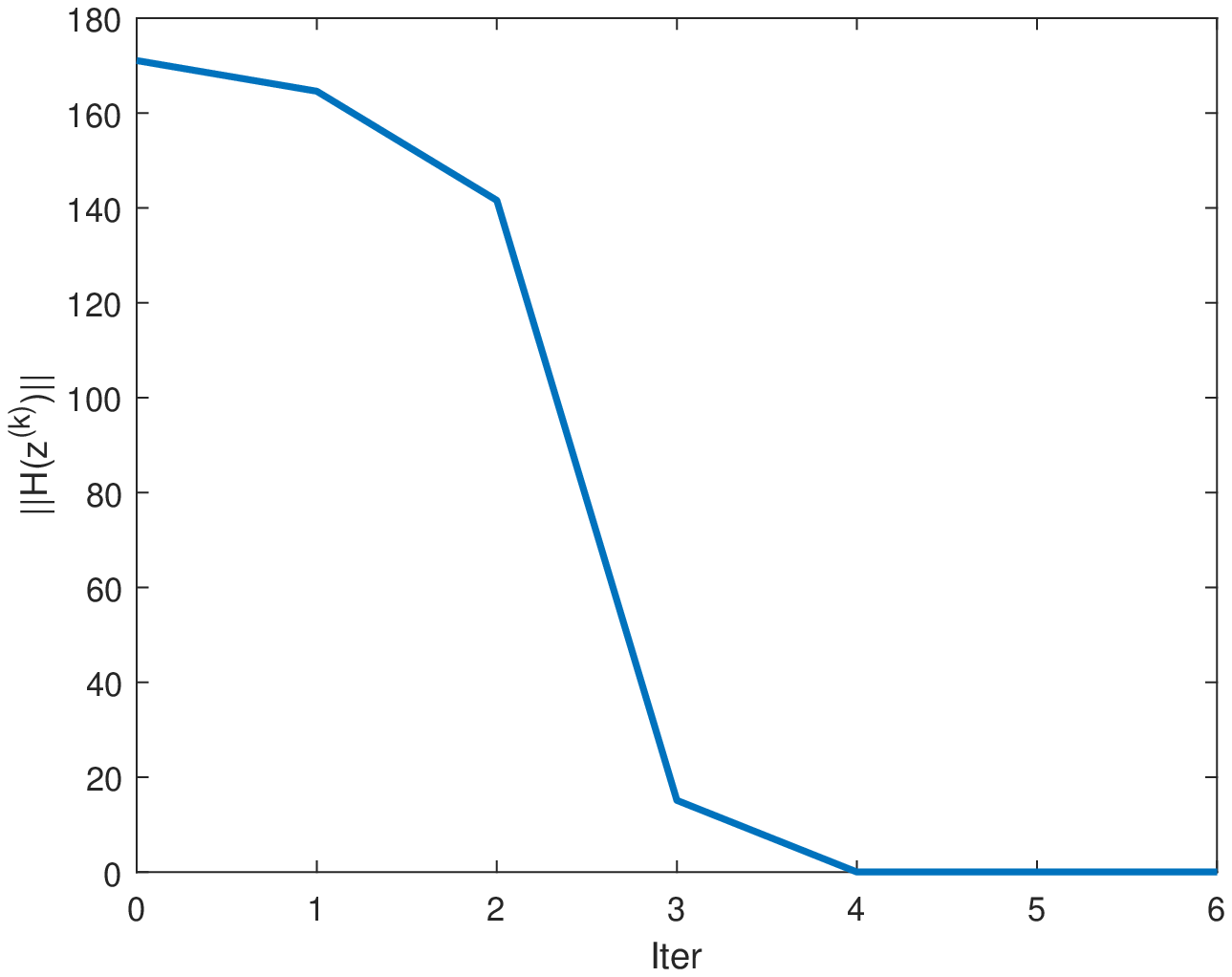}} \vspace{2ex}\\
\hspace{-0.3 cm}
\resizebox*{0.3\textwidth}{0.26\textheight}{\includegraphics{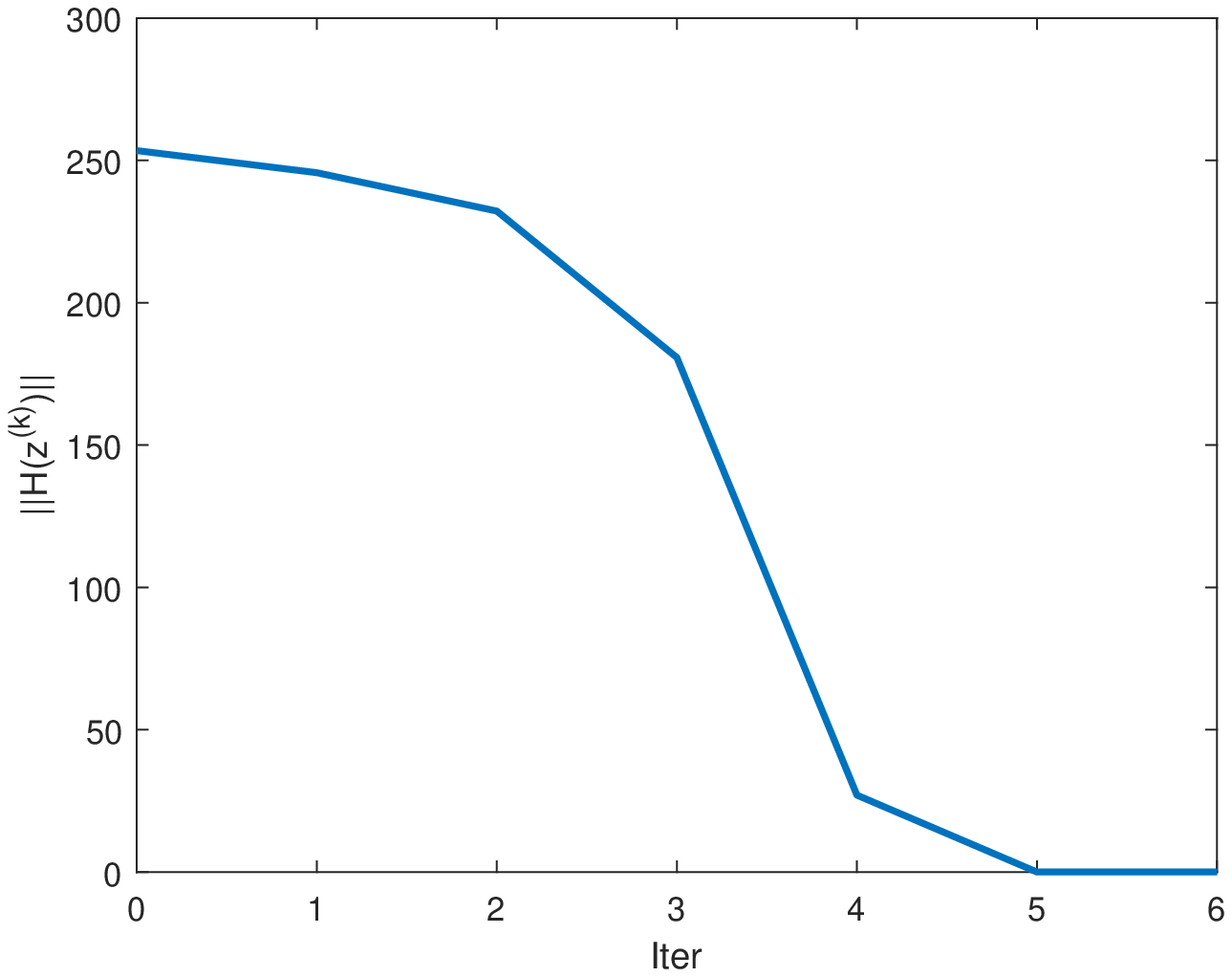}}
&\hspace{-0.4 cm}\resizebox*{0.4\textwidth}{0.26\textheight}{\includegraphics{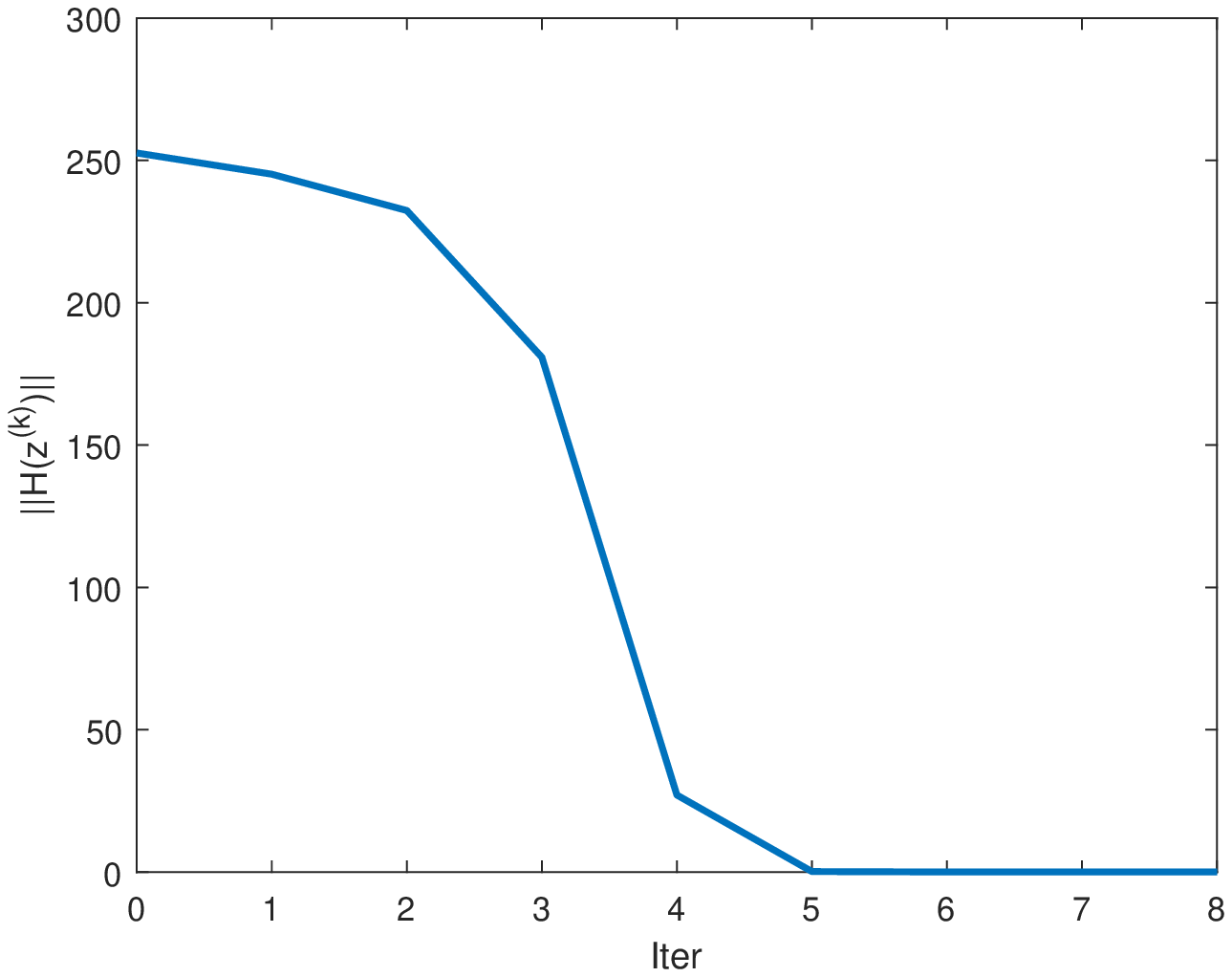}} & \hspace{-0.4 cm}
\resizebox*{0.3\textwidth}{0.26\textheight}{\includegraphics{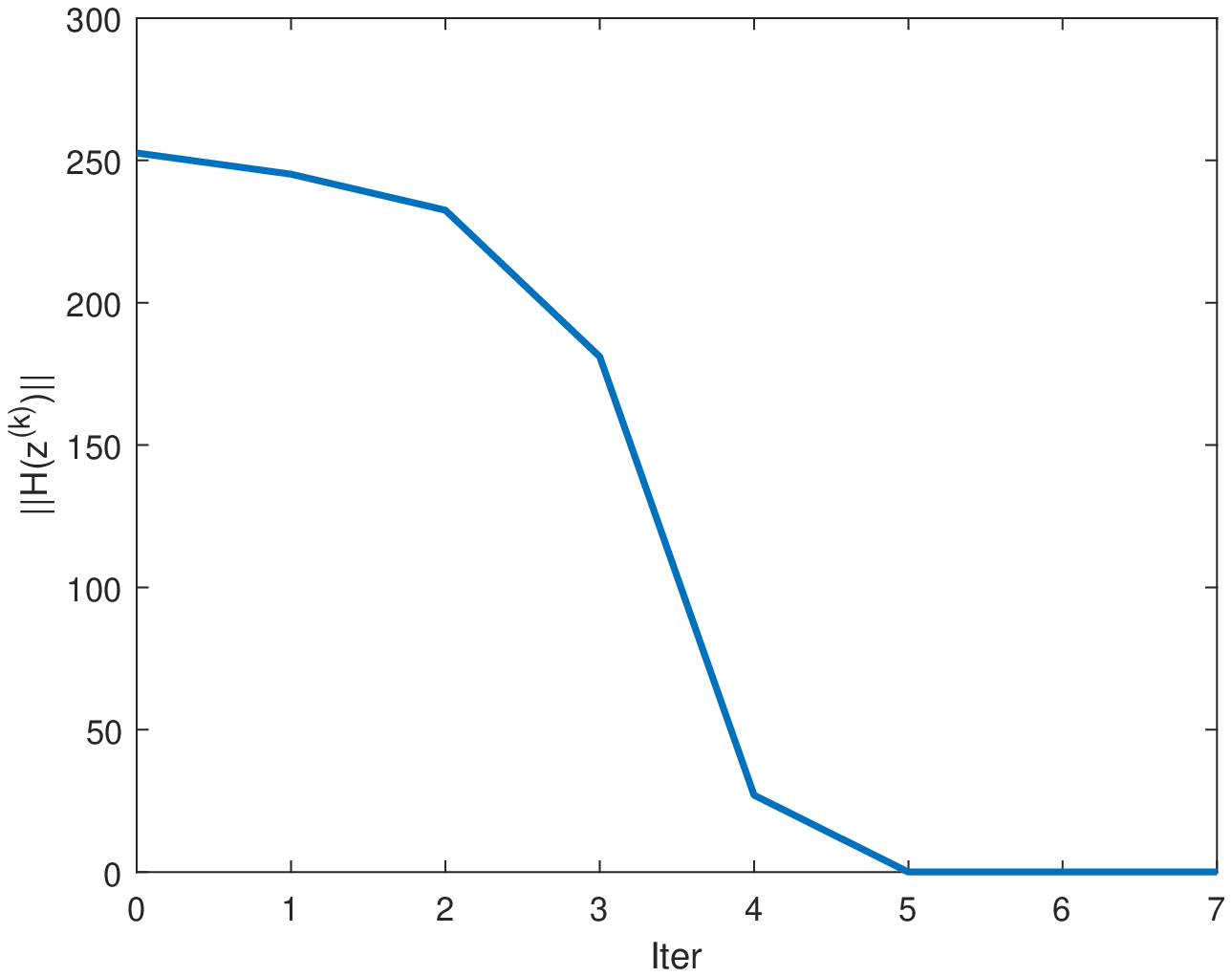}} \vspace{2ex}\\
\hspace{-0.3 cm}
\resizebox*{0.3\textwidth}{0.26\textheight}{\includegraphics{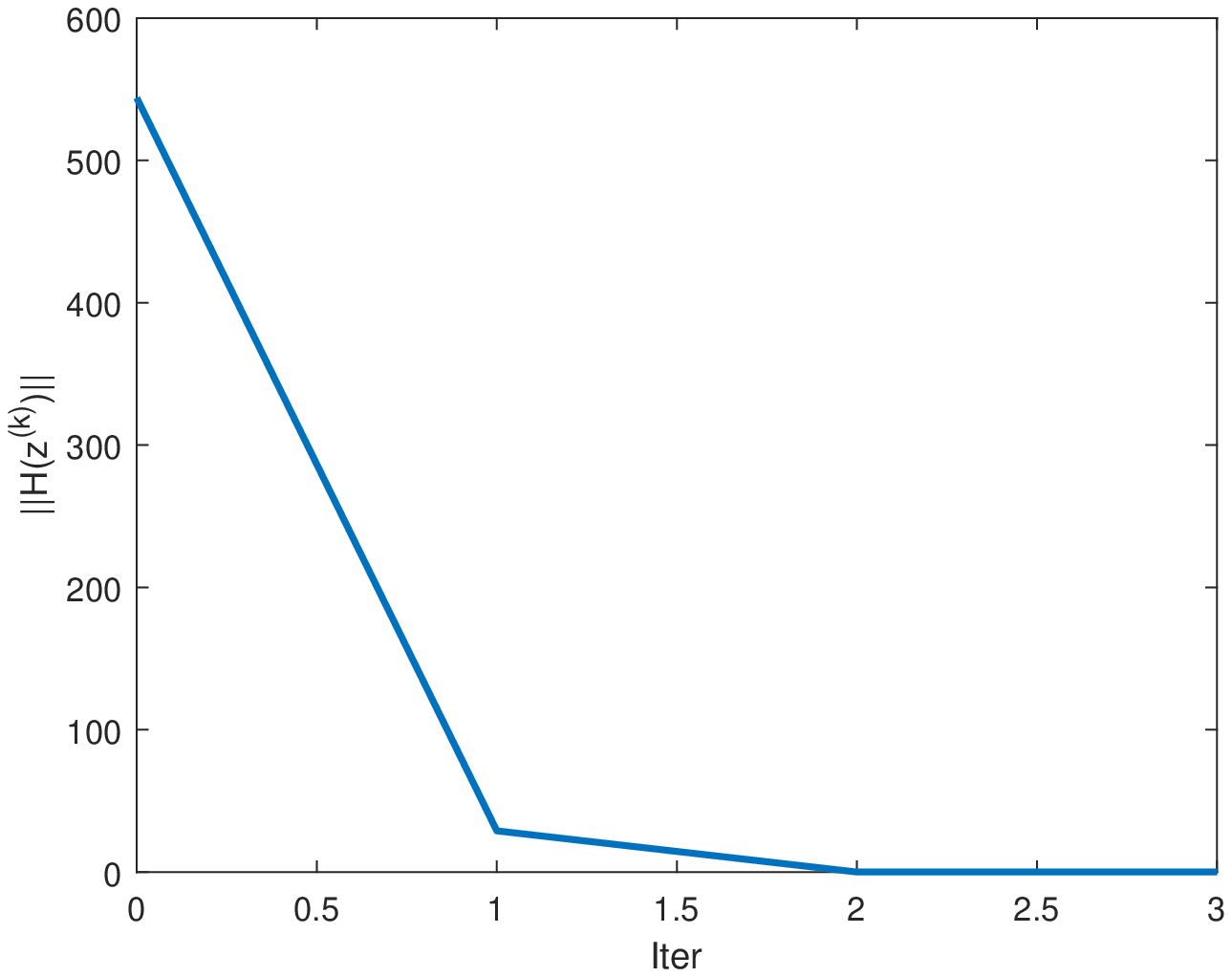}}
&\hspace{-0.4 cm}\resizebox*{0.4\textwidth}{0.26\textheight}{\includegraphics{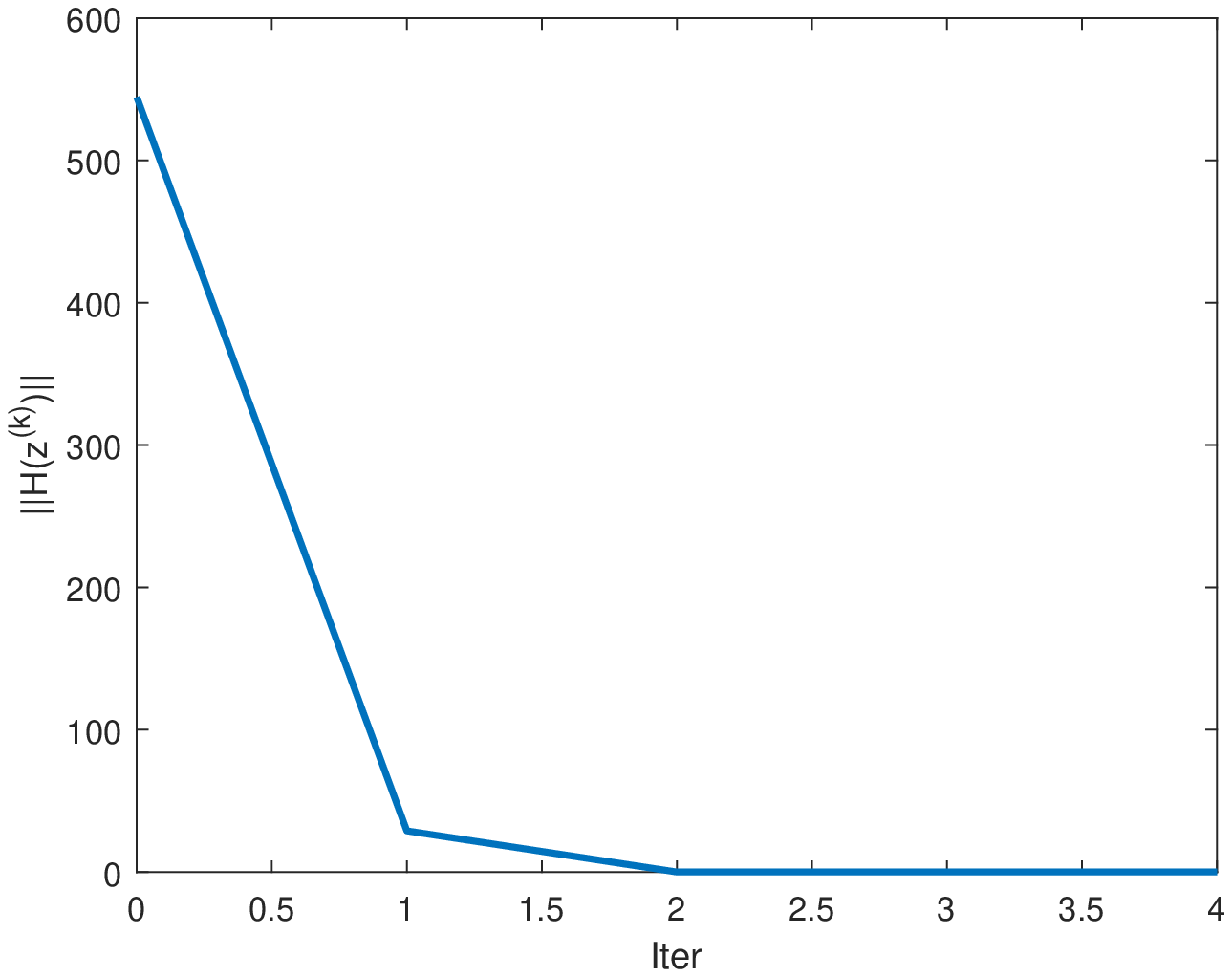}} & \hspace{-0.4 cm}
\resizebox*{0.3\textwidth}{0.26\textheight}{\includegraphics{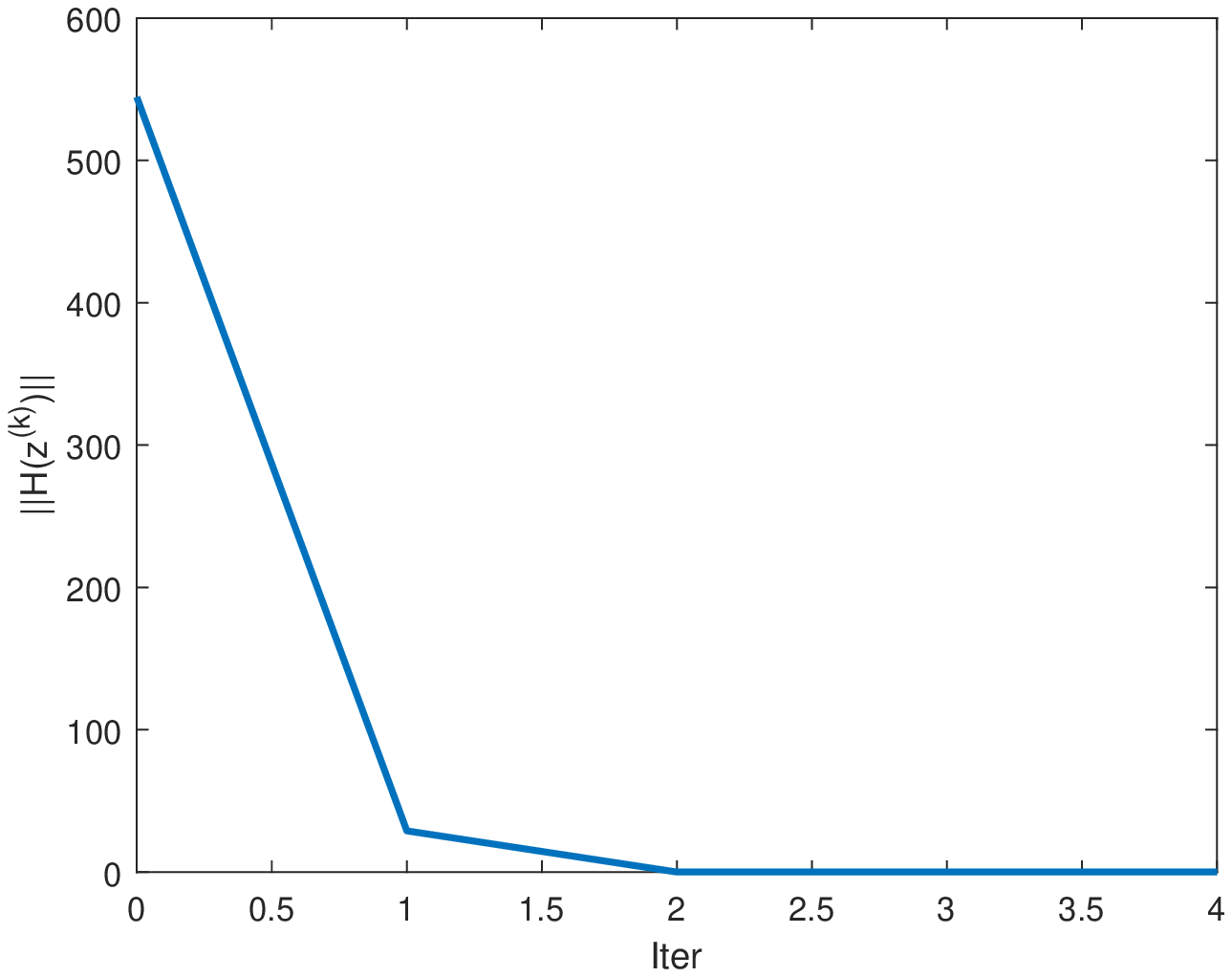}} \vspace{2ex}\\
\end{tabular}\par
}\vspace{-0.15 cm}
\caption{Convergence history curves for Example~\ref{exam:exam1} with $n = 32^2$. The plots in the first column are for NSNA, the plots in the second column are for JZ-MSNA and the plots in the third column are for TZ-MSNA, respectively. The plots in the first row are for $\xi = \zeta =0$, the plots in the second row are for $\xi = 0$ and $\zeta = 4$ and the plots in the third row are for $\xi = 4$ and $\zeta = 0$, respectively.
}
\label{Fig:Curves-for-P1}
\end{figure}

\begin{figure}[t]
{\centering
\begin{tabular}{ccc}
\hspace{-0.3 cm}
\resizebox*{0.3\textwidth}{0.26\textheight}{\includegraphics{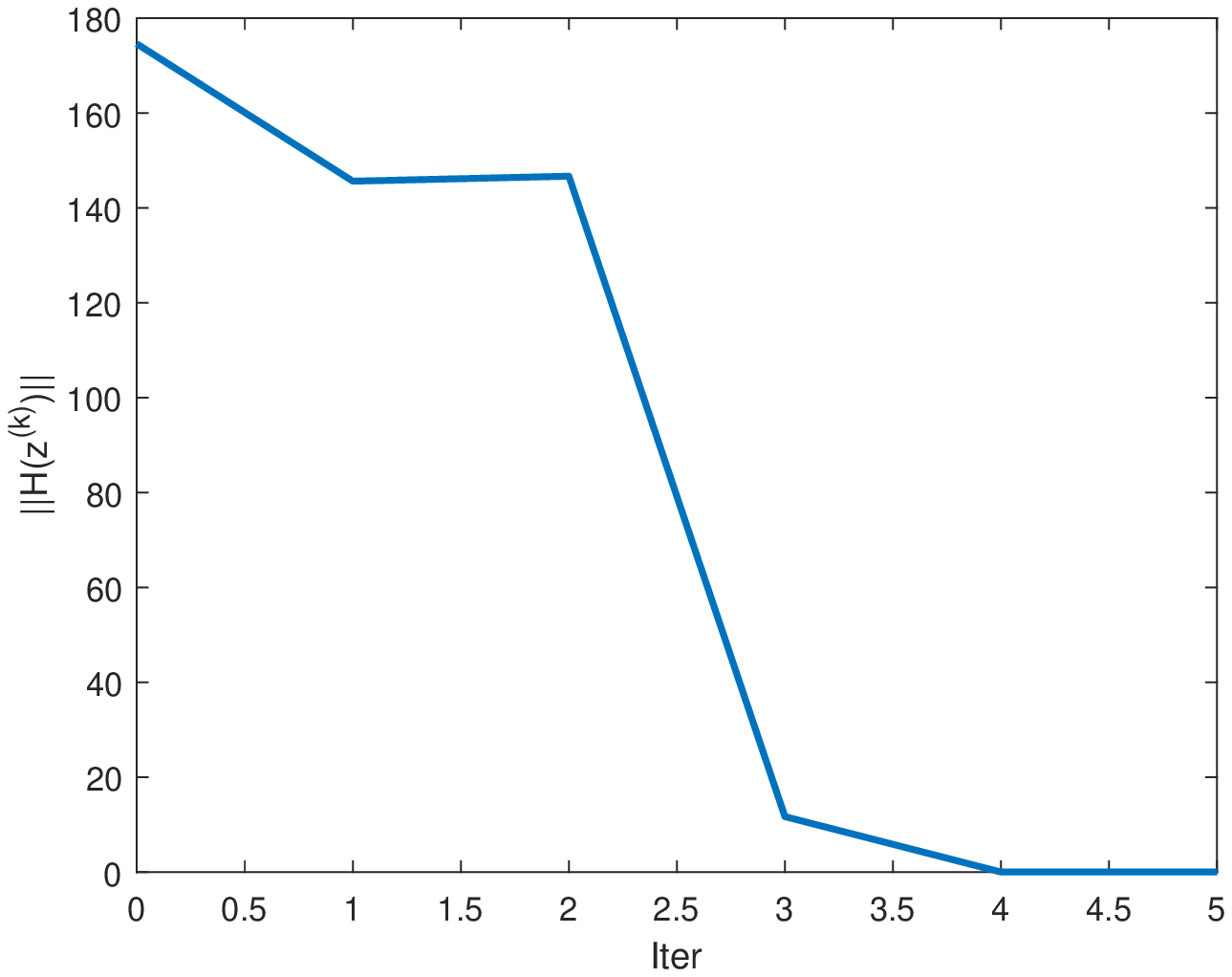}}
&\hspace{-0.4 cm}\resizebox*{0.4\textwidth}{0.26\textheight}{\includegraphics{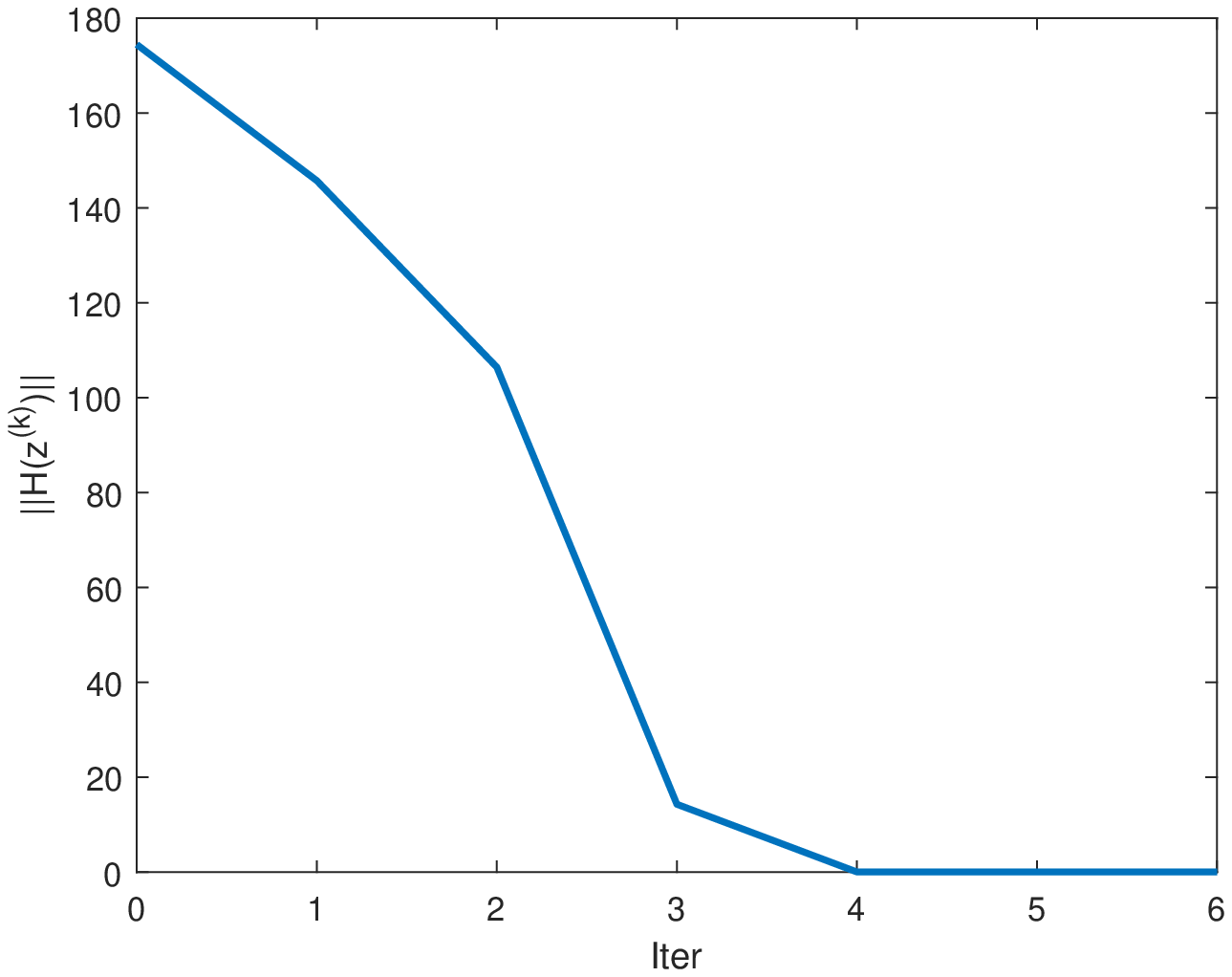}} & \hspace{-0.4 cm}
\resizebox*{0.3\textwidth}{0.26\textheight}{\includegraphics{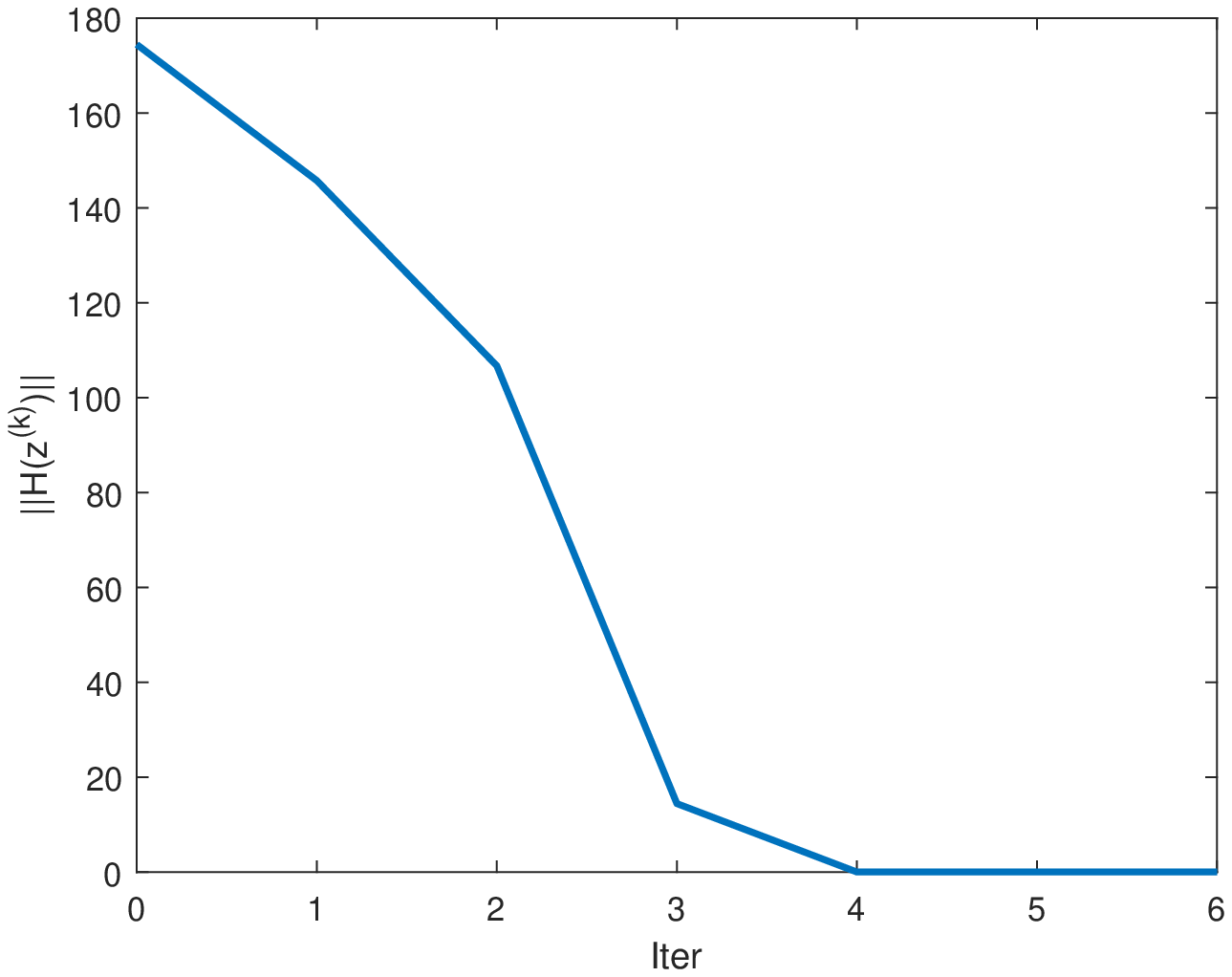}} \vspace{2ex}\\
\hspace{-0.3 cm}
\resizebox*{0.3\textwidth}{0.26\textheight}{\includegraphics{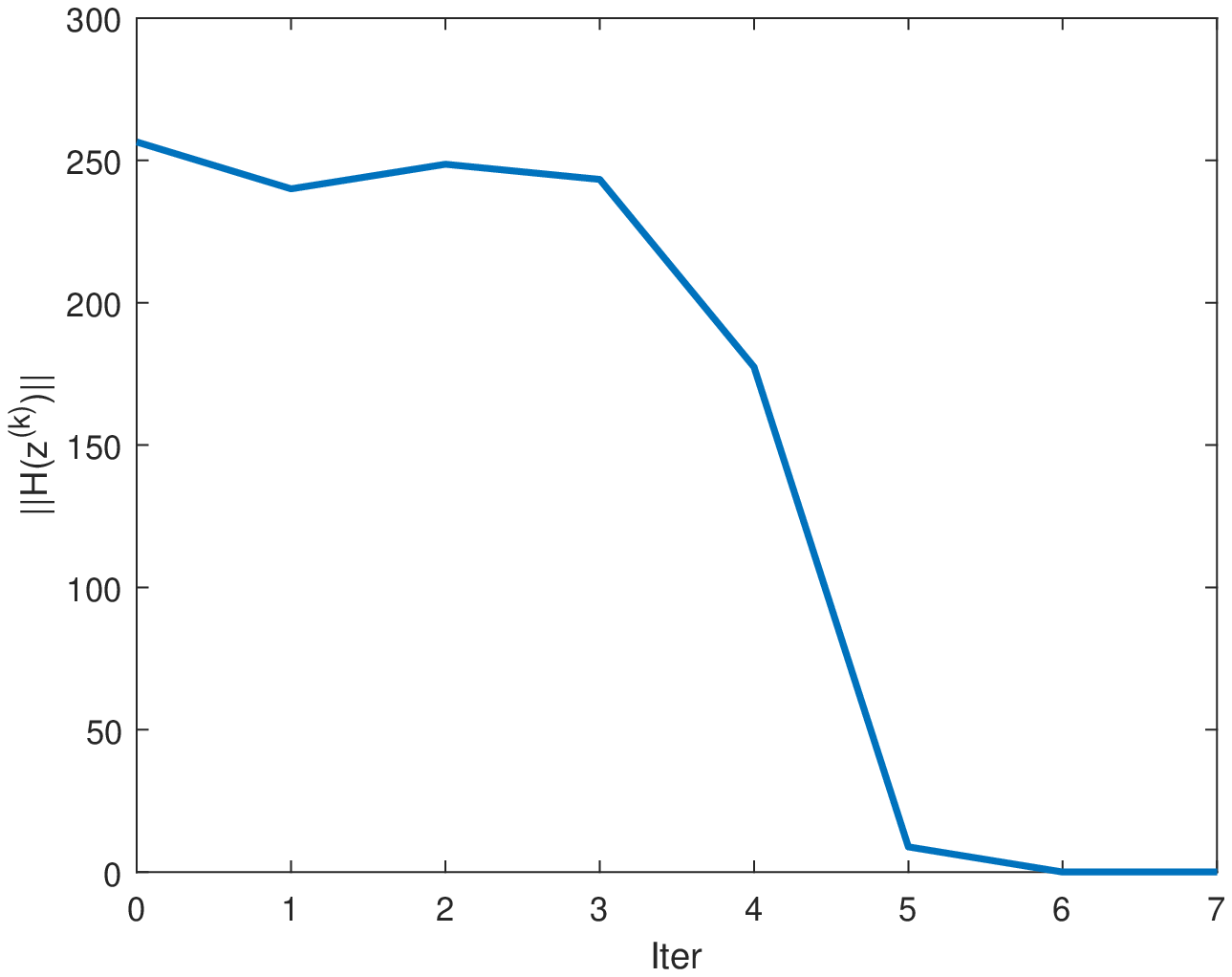}}
&\hspace{-0.4 cm}\resizebox*{0.4\textwidth}{0.26\textheight}{\includegraphics{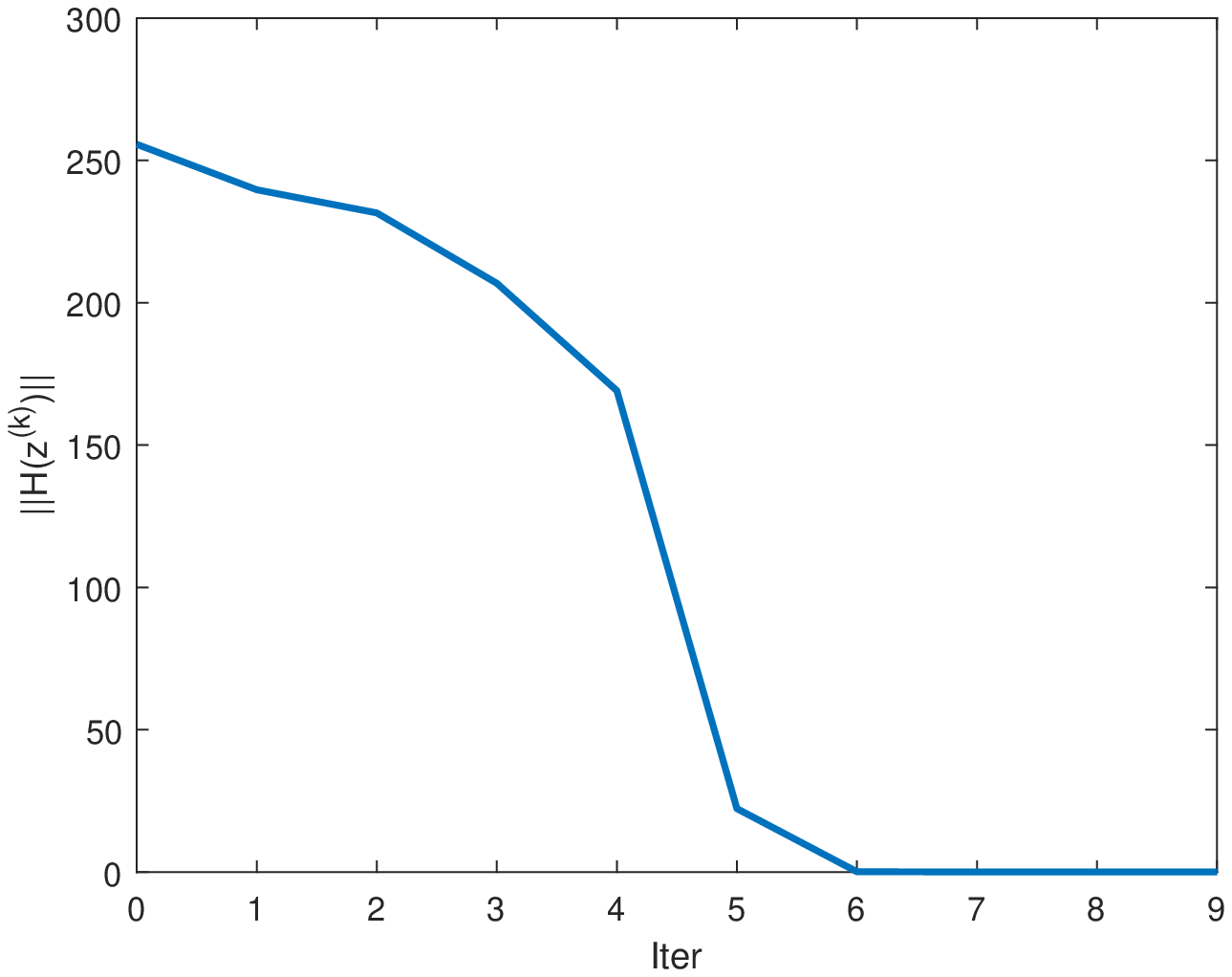}} & \hspace{-0.4 cm}
\resizebox*{0.3\textwidth}{0.26\textheight}{\includegraphics{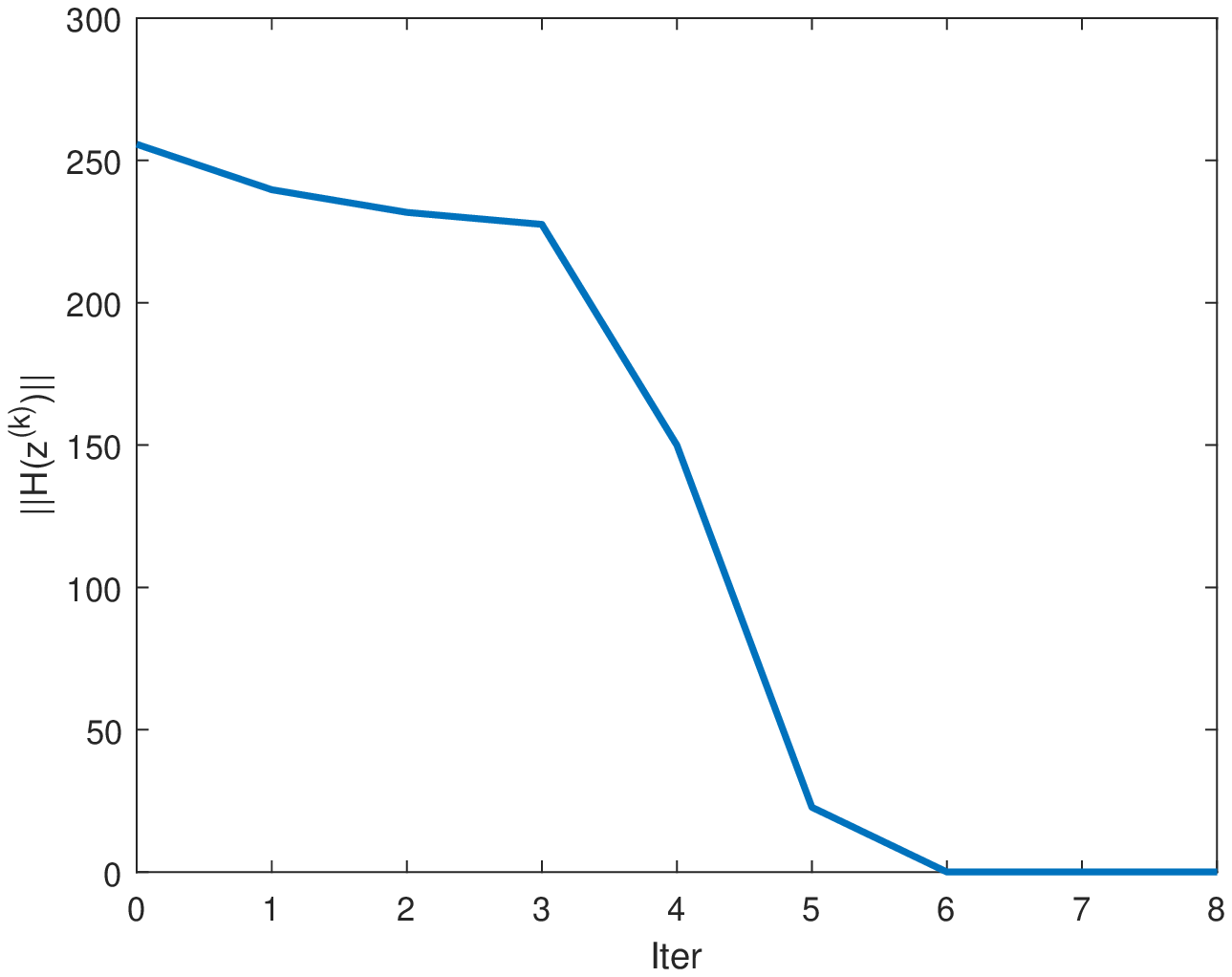}} \vspace{2ex}\\
\hspace{-0.3 cm}
\resizebox*{0.3\textwidth}{0.26\textheight}{\includegraphics{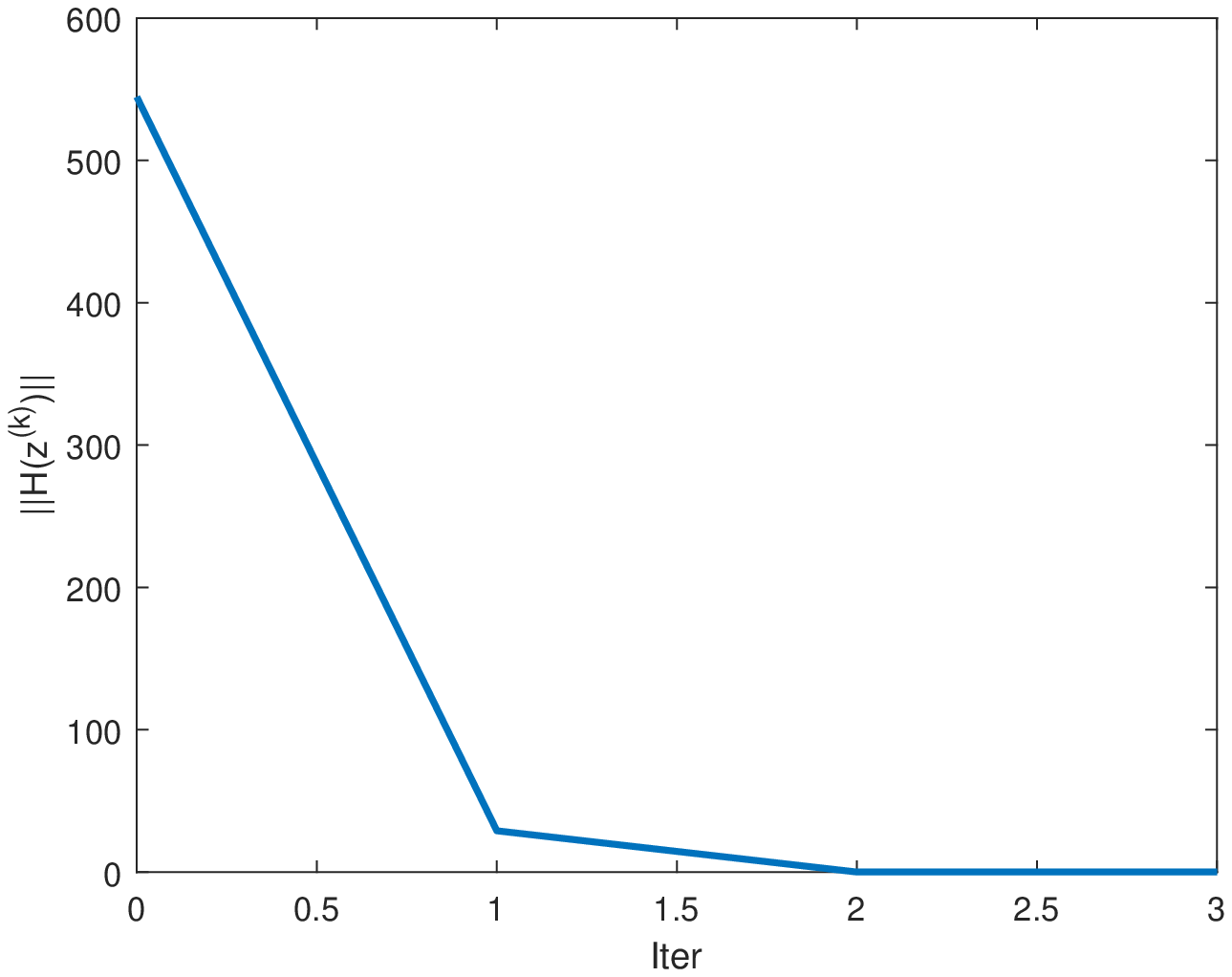}}
&\hspace{-0.4 cm}\resizebox*{0.4\textwidth}{0.26\textheight}{\includegraphics{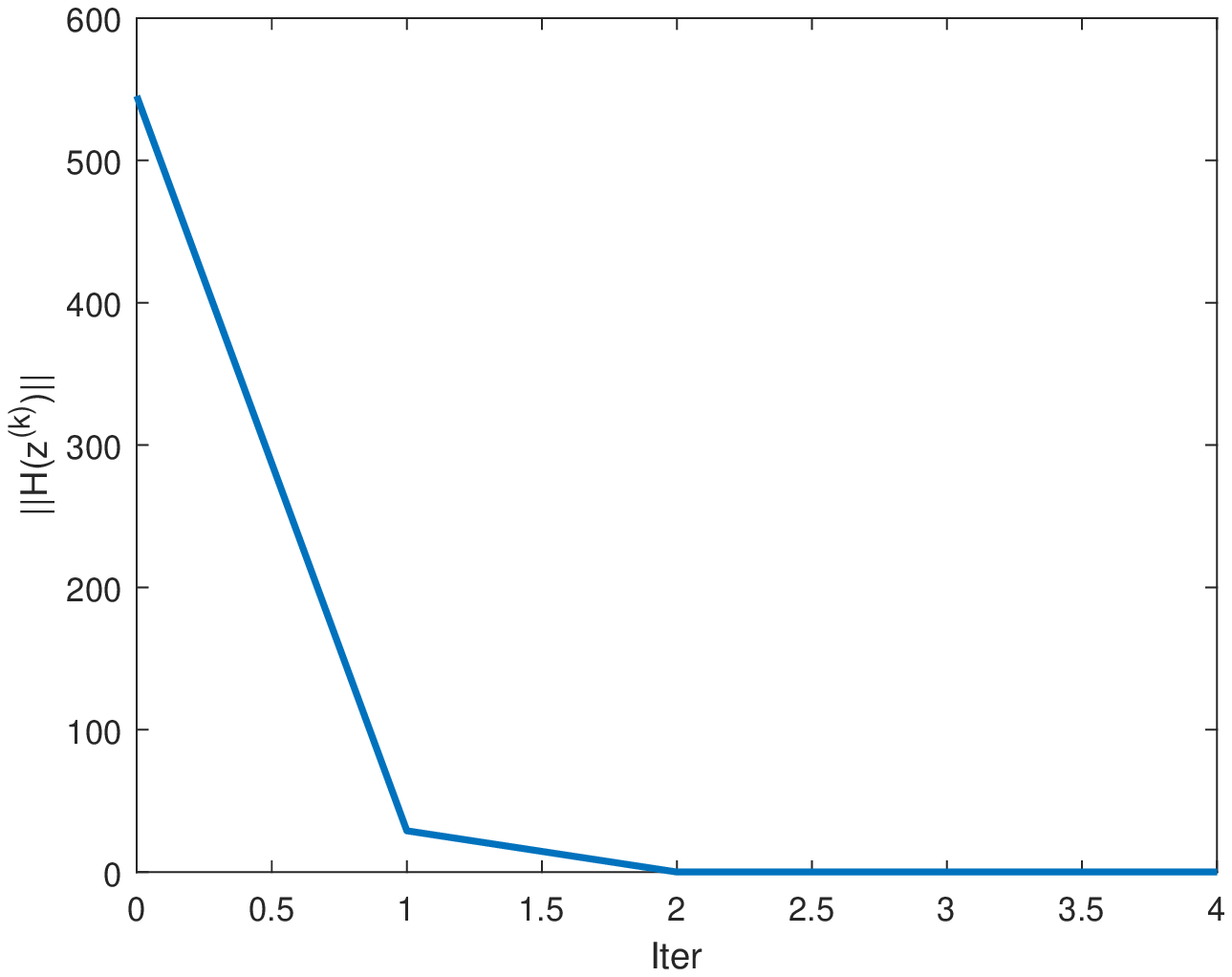}} & \hspace{-0.4 cm}
\resizebox*{0.3\textwidth}{0.26\textheight}{\includegraphics{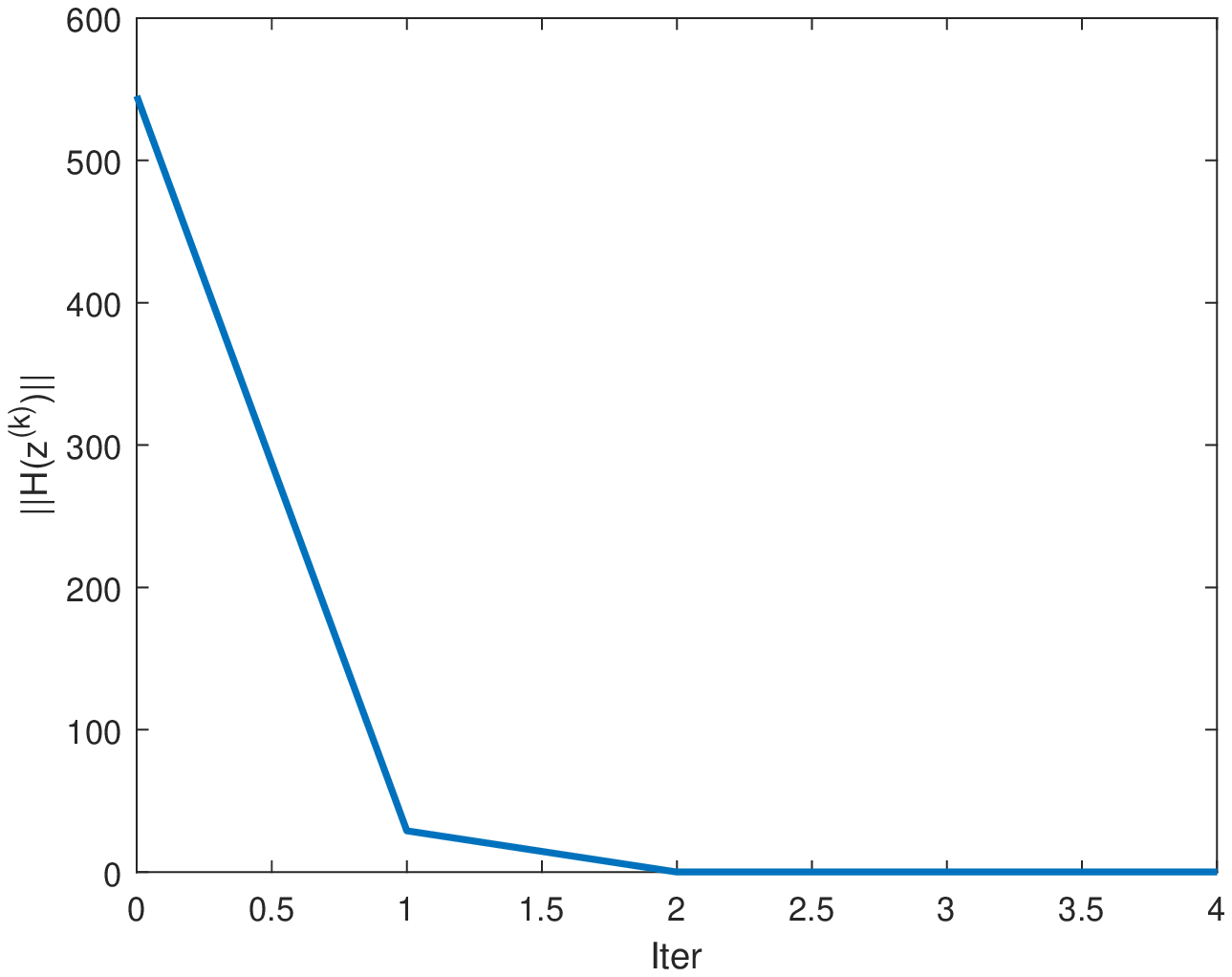}} \vspace{2ex}\\
\end{tabular}\par
}\vspace{-0.15 cm}
\caption{Convergence history curves for Example~\ref{exam:exam2} with $n = 32^2$. The plots in the first column are for NSNA, the plots in the second column are for JZ-MSNA and the plots in the third column are for TZ-MSNA, respectively. The plots in the first row are for $\xi = \zeta =0$, the plots in the second row are for $\xi = 0$ and $\zeta = 4$ and the plots in the third row are for $\xi = 4$ and $\zeta = 0$, respectively.
}
\label{Fig:Curves-for-P2}
\end{figure}

%%%%%%%%%%%%%%%%%%%%%%%%%%%%%%%%%%%%%%%%%%
\section{Conclusions}\label{sec:conclusion}
In this paper, a non-monotone smoothing Newton method is proposed to solve the system of generalized absolute value equations. Under a weaker assumption, we prove the global and the local quadratic convergence of our method. Numerical results demonstrate that our method can be superior to two existing methods in some situations.

%\section*{Acknowledgements}

%%%%%%%%%%%%%%%%%%%%%%%%%%%%%%%%%%%%%%%%%%%%

\end{document}